\documentclass[preprint]{imsart}
\usepackage[sort]{natbib}
\usepackage[title,titletoc]{appendix}

\usepackage{microtype}
\usepackage{graphicx}
\usepackage{subfigure}
\usepackage{booktabs}
\usepackage{amsfonts,amsmath,amssymb,amsthm,bbm}
\usepackage[table]{xcolor}
\usepackage{nameref,zref-xr}

\usepackage{xr}
\newcommand{\bk}{\color{black}}

\newcommand{\Var}{\text{Var}}
\newcommand{\vj}{\widehat{\mathrm{Var}}_{\mathrm{JACK}}}

\newcommand{\aibar}{\bar{\alpha}_n}
\newcommand{\aibarj}{\bar{\alpha}_{nj}}
\newcommand{\bibar}{\bar{\beta}_n}
\newcommand{\Aibar}{\bar{A}_n}
\newcommand{\Bibar}{\bar{B}_n}
\newcommand{\cp}{\stackrel{P}{\rightarrow}}

\usepackage{hyperref}

\newcommand{\cov}{\mathrm{cov}}
\newcommand{\var}{\mathrm{Var}}
\newcommand{\df}{\nabla f(\mu)}

\newenvironment{proofsketch}{%
  \renewcommand{\proofname}{Proof Sketch}\proof}{\endproof}
\newcommand{\norm}[1]{\left\lVert#1\right\rVert}
\newtheorem{theorem}{Theorem}
\newtheorem*{theorem*}{Theorem}

\newtheorem{lemma}[theorem]{Lemma}
\newtheorem{proposition}{Proposition}
\newtheorem{remark}{Remark}

\begin{document}

\begin{frontmatter}
\title{On the Theoretical Properties of the Network Jackknife}
\runtitle{}

\begin{aug}
\author{\fnms{Qiaohui} \snm{Lin,} \ead[label=e1]{qiaohui.lin@utexas.edu}}
\author{\fnms{Robert} \snm{Lunde,}\ead[label=e1]{rlunde@utexas.edu}}
and
\author{\fnms{Purnamrita Sarkar}\ead[label=e2]{purna.sarkar@austin.utexas.edu}}
\affiliation{The University of Texas at Austin}

\runauthor{Q.Lin,R.Lunde and P.Sarkar}
\runtitle{Jackknifing Networks}



\address{Department of Statistics and Data Science\\
University of Texas at Austin}

\end{aug}

\begin{abstract}
We study the properties of a leave-node-out jackknife procedure for network data. Under the sparse graphon model, we prove an Efron-Stein-type inequality, showing that the network jackknife leads to conservative estimates of the variance (in expectation) for any network functional that is invariant to node permutation. For a general class of count functionals, we also establish consistency of the network jackknife. We complement our theoretical analysis with a range of simulated and real-data examples and show that the network jackknife offers competitive performance in cases where other resampling methods are known to be valid. In fact, for several network statistics, we see that the jackknife provides more accurate inferences compared to related methods such as subsampling.
\end{abstract}

\end{frontmatter}

\section{Introduction}
\label{sec:intro}
 Network-structured data are now everywhere. The Internet is a giant, directed network of webpages pointing to other webpages. Facebook is an undirected network  built via friendships between users. The ecological web is a directed network of different species with edges specified by `who-eats-whom' relationships.  Protein-protein interactions are undirected networks consisting of pairs of bait-prey proteins that bind to each other during coaffinity purification experiments arising in mass spectrometry analysis.   
 
 
 In these application areas,  it is often of interest to characterize a network using statistics such as the clustering coefficient, triangle density, or principal eigenvalues. There has been a substantial amount of work on approximating these quantities with small error on massive networks~\cite{SUBG:Feige-average-degree,SUBG:Goldreich-Ron-Approximating-Average-Parameters,SUBG:assadi-sublinear-arbitrary-subgraph,SUBG:eden-sublinear-triangles,SUBG:gonen-sublinear-stars,Kallaugher:2019:CCC:3294052.3319706}. However, comparatively little attention has been paid to assessing the variability of these statistics with a few exceptions that we will discuss shortly.  Quantifying the uncertainty of these estimators is of utmost importance, as it gives us information about the underlying variability of the data generating process. Take for example the problem of comparing two networks, which is a key question in many biological applications and in social network analysis.  A natural direction  would be to first obtain resamples of networks to construct distributions of different summary statistics and then compare these distributions. While there has been some recent interest in two-sample tests for networks~\cite{kim2014,durante2018,two-sample-test-network-statistics,tang-nonparametric-two-sample-test}, very few works use resampling to compare networks. 
 
 

Resampling methods have a long and celebrated history in statistics, with the bootstrap, jackknife,and subsampling being the three main forms. There is a now vast literature developing these methods for iid data;  for seminal works in this area, see  \cite{quenouille-jackknife,efron1986,Bickel-m-out-of-n,Politis-Romano-Wolf-subsampling,shao-wu-jackknife-variance}. 
Even when the data are not independent, resampling methods have been shown to yield asymptotically valid inferences for a wide range of functionals under various dependence structures. For weakly dependent time series, for example, the key innovation is to resample contiguous blocks of data instead of individual observations. Under mild conditions on the block length and nature of dependence, blocked variants of resampling methods, including the block bootstrap \citep{Kunsch-bootstrap-for-stationary-obs}, block subsampling \citep{politis-romano-subsampling-minimal-assumptions}, and the blockwise jackknife \citep{Kunsch-bootstrap-for-stationary-obs} have been shown to asymptotically capture the dependence structure of the data, leading to theories that closely resemble the corresponding theories for iid data.

Recently, some work has started to emerge involving resampling procedures for networks. \citet{levin-levina-rdpg-bootstrap} propose two bootstrap procedures for random dot product graphs  that involve estimating the latent positions and resampling the estimated positions to conduct inference for the functional of interest.  The authors establish bootstrap consistency for functionals that are expressible as U-statistics of the latent positions, which encompasses many important classes of functionals including subgraph counts.  

\citet{subsampling-sparse-graphons} consider a procedure that involves subsampling nodes and computing functionals on the induced subgraphs. This procedure is shown to be asymptotically valid under conditions analogous to the iid setting; that is, the subsample size must be $o(n)$ and  the functional of interest must converge to a non-degenerate limit distribution. Previously, \citet{Bhattacharyya-subsample-count-features} had shown the validity of subsampling for count functionals. By proving a central limit theorem for eigenvalues, \citet{subsampling-sparse-graphons} also establish subsampling validity for these functionals under certain conditions.  Finally, in \citet{green2017boot}, the authors propose sieve and nonparametric bootstrap procedures for networks. 

We would like to note that that both the sieve approach of \citet{green2017boot} and the latent position estimation approaches of \citet{levin-levina-rdpg-bootstrap} depend on accurately estimating the underlying graphon. The nonparametric bootstrap procedure described in \citet{green2017boot} requires resampling much larger networks from a size $n$ network, leading to computational inefficiency. Even subsampling requires weak convergence and a known rate of convergence; it turns out that the latter may be estimated \citep{bertail-politis-romano-subsampling-estimated-rate}, but doing so entails a substantial increase in computation and is likely to adversely affect the finite-sample performance of the procedure. While asymptotically valid under general conditions, the finite-sample performance of subsampling methods is known to be sensitive to the choice of tuning parameters; see for example, \cite{blbjrssb}. 

\subsection{Our Contribution}
In the present work, we study the properties of a network jackknife introduced by \citet{franks-snijders-hidden-pop-snowball-sampling} under the sparse graphon model. On the theoretical side, we make two primary contributions.  First, analogous to the iid setting, we show that the network jackknife produces variance estimates that are conservative in expectation under general conditions.  Our result here justifies the network jackknife as a rough-and-ready tool that produces reasonable answers (erring on the side of caution) even when the asymptotic properties of the functional of interest are poorly understood.

While the upward bias of the network jackknife is a favorable property, it does not provide information as to how the jackknife compares to other resampling procedures for more well-understood functionals.  As another theoretical contribution, we establish consistency of the jackknife for a general class of count statistics studied in \citet{Bickel-Chen-Levina-method-of-moments}.  We also extend our result to smooth functions of counts, which encompasses widely used measures such as the transitivity coefficient.

We complement our theoretical results with an empirical investigation of the network jackknife on both simulated and real datasets.  In our simulation study, we study the rate of convergence of the jackknife variance estimate for two sparse graphon models under a range of choices for the network functional. Our results suggest that by and large, the jackknife has better finite-sample properties than subsampling. For real data, we conduct network comparisons of Facebook networks constructed from a number of different colleges such as Caltech, Berkeley, Stanford, Wellesley, etc. 

The paper is organized as follows. In Section~\ref{sec:background}, we do problem setup and introduce notation. In Section~\ref{sec:theory}, we present our theoretical results and some proof sketches. Finally in Section~\ref{sec:exp} we present experimental results on simulated and real networks.

\section{Background}
\label{sec:background}
We first briefly recall the original Jackknife for iid data, followed by a description of the sparse graphon model. We conclude this section with the network jackknife procedure.
\subsection{The Jackknife for IID Data}
The jackknife, attributed to \citet{quenouille-jackknife} and \citet{Tukey-jackknife}, is a resampling procedure that involves aggregating leave-one-out estimates.  More precisely, let $X_1, \ldots, X_n \sim P$ and let $S_n$ be a permutation-invariant function of $n$ variables.  Furthermore, let $S_{n,i}$ denote the functional computed on the dataset with $X_i$ removed and let $\bar{S}_{n} = \frac{1}{n}\sum_{i=1}^n S_{n,i}$. The jackknife estimate of the variance of $S_{n-1} = S(X_1, \ldots, X_{n-1})$ is given by:
  \begin{align}
\vj \ S_{n-1} := \sum_{i=1}^n (S_{n,i} - \bar{S}_{n})^2
\end{align} 

%
%
For appropriately smooth functionals, it is well-known that the jackknife consistently estimates the variance; see for example, \citet{shao-tu-bootstrap-jackknife}. The bootstrap, introduced by \citet{efron-jackkknife-bootstrap}, typically requires weaker regularity conditions than the jackknife for consistency. 
In fact, it is well-known that the jackknife is inconsistent for the median \citep{miller-jackknife} while the bootstrap variance remains consistent under reasonable conditions \citep{ghosh1984}\footnote{It should be noted that the delete-d jackknife is valid under more general conditions; see \citet{shao-wu-jackknife-variance}.}. 

However, for more complicated functionals, it may often be the case that both the bootstrap and the jackknife are inconsistent\footnote{Recent work by \citet{fang-santos-directionally-differentiable} suggests that Hadamard differentiability of $g$ is both necessary and sufficient for bootstrap consistency of $g(\hat{\theta}_n)$ whenever $\hat{\theta}_n$ is asymptotically Gaussian.}.
Even in these cases, the jackknife still provides reasonable answers. The remarkable inequality of \citet{efron-stein-jackknife} asserts that the jackknife is always upwardly biased, ensuring a conservative estimate of the variance. 


Since networks are inherently high-dimensional objects, asymptotic results are often harder to come by compared to the iid setting.  The theory of the jackknife for iid processes suggests that, even in this challenging regime, a network analogue of the jackknife may have some advantageous properties.  Before delving into our method, we introduce the network model under consideration below.
\subsection{The Sparse Graphon Model} 
The model below is due to \citet{Bickel-Chen-on-modularity}. Let $\{A^{(n)}\}_{n \in \mathbb{N}}$ denote a sequence of $n \times n$ adjacency matrices and $\xi_1, \xi_2, \ldots \xi_n \sim \mathrm{Uniform}[0,1]$.  Furthermore, for $i \neq j$, let  $\eta_{ij} \sim \mathrm{Uniform}[0,1]$ and define $A_{ij}^{(n)}$ as follows:
\begin{align}
\label{eq:sparse-graphon}
\begin{split}
    A_{ij}^{(n)} = A_{ji}^{(n)} &=\mathbbm{1}(\eta_{ij}\leq \rho_n w(\xi_i, \xi_j) \wedge 1)\sim \mathrm{Bernoulli}(\rho_n w(\xi_i, \xi_j) \wedge 1) 
\end{split}
\end{align} 
\bk
where $w:[0,1]^2 \mapsto \mathbb{R}_{\geq 0}$ is a symmetric measurable function that satisfies $\int_0^1 \int_0^1 w(u,v) \ du \ dv = 1$ and $\rho = \{\rho_n\}_{n \in \mathbb{N}} \in [0,1]^{\mathbb{N}}$.  When $\norm{\rho_n w(u,v)}_\infty \leq 1$, $\rho_n$ may be interpreted as the probability of an edge and $w$ as the conditional density of the latent positions given an edge.   For $i = j$, we will assume that $A_{ii}^{(n)} = 0$. We refer to $w$ as a graphon, the pair $(\rho, w)$ as a sparse graphon, and (\ref{eq:sparse-graphon}) as graphs generated by the sparse graphon model.

Bounded graphons arise as a limiting object in the theory of graph limits; see \citet{Lovacz-Graph-Limits}. Alternatively, graphons are a natural representation for (infinite-dimensional) jointly exchangeable arrays, where this notion of exchangeability corresponds to invariance under vertex permutation; see for example, \citet{diaconis-janson-exchangeable-graph}. Bounded graphons subsume many other commonly studied network models, including stochastic block models \citep{holland-sbm} and random dot product graphs \citep{young-schneiderman-rdpg}.

For applications, one major issue with graphons is that they generate dense graphs, where the expected number of edges is $O(n^2)$. However, most real-world graphs are known to be sparse in the sense that the number of edges is $o(n^2)$.  Here, the sequence $\rho$ determines the sparsity level; by letting $\rho_n \rightarrow 0$ at an appropriate rate, we may generate a graph sequence with the desired sparsity properties.  Furthermore, notice that in our formulation above, $w$ is allowed to be unbounded.  As noted in \citet{borgs-lp-part-one}, unbounded graphons are more expressive than bounded graphons, allowing graphs that exhibit power law degree distributions.

\subsection{The Network Jackknife Procedure}                        
Let $f:\{0,1\}^{ \ n-1 \ \times \ n-1} \mapsto \mathbb{R}$ denote a function that takes as input a $n-1\times n-1$ adjacency matrix and let $Z_{n,i}$ denote the random variable formed by applying $f$ to an induced subgraph with node $i$ removed. Under the model (\ref{eq:sparse-graphon}), observe that each induced subgraph formed by leaving a node out is identically distributed as a consequence of vertex exchangeability.  Therefore, functionals calculated on these induced subgraphs are similar in spirit to the the leave-one-out estimates for the jackknife in the iid setting.  Following \citet{franks-snijders-hidden-pop-snowball-sampling}, a natural generalization of the jackknife to the sparse graphon setting is given by:
\begin{align}
\vj \ Z_{n-1} := \sum_{i=1}^n (Z_{n,i} - \bar{Z}_{n})^2
\end{align}
where $\bar{Z}_n = \frac{1}{n}\sum_{i=1}^n Z_{n,i}$ and  $\vj\   Z_{n-1}$ is an estimate of $\mathrm{Var} \  Z_{n-1}$, the variance with respect to an induced subgraph with node set $\{1,\ldots, n-1\}$. We would like to note that letting $Z_{n-1} := Z_{n,n}$ constitutes a slight abuse of notation since $\rho_{n-1}$ need not equal $\rho_n$, but doing so substantially improves the readability of our proofs.

\section{Theoretical Results}
\label{sec:theory}
\subsection{The Network Efron-Stein Inequality}
The first result we state here is our generalization of the Efron-Stein inequality to the network setting.  Intuitively, the Efron-Stein inequality may be thought of as a general property of functions of independent random variables. While edges in the adjacency matrix are dependent through the latent positions, the fact that they are functions of independent random variables allow us to prove the following:

\begin{theorem}[Network Efron-Stein Inequality]
\label{thm:netefron}
\begin{align}
\label{eq:network-efron-stein}
\mathrm{Var} \ Z_{n-1} \leq E( \vj \ Z_{n-1})
\end{align} 
\end{theorem}
The main ingredients in our proof are an adaptation of a martingale argument due to \citet{rhee-talagrand-martingale-approach-jackknife} and an appropriate filtration for graphon models inspired by \citet{borgs-convergence-dense-graphons-1}.  We provide a proof sketch below; for details, see Supplement Section~\ref{sec:suppES}.
\begin{proofsketch}
As discussed in the Supplementary Material, for $1 \leq i \leq n$, we may express $Z_{n,i}$ as a measurable function of latent positions $\xi_i \sim \mathrm{Unif}[0,1]$ for $1 \leq i \leq n$ and $\eta_{ij} \sim \mathrm{Unif}[0,1]$ for $1 \leq i < j \leq n$.
More precisely, $Z_{n,i}$ is a function of the variables that are not shaded below:
\begin{equation}
     Z_{n,i}=g\left(
\begin{array}{cccc>{\columncolor{gray!20}}ccc}
     \xi_1 &\eta_{12} &\eta_{13} &... &\eta_{1i} &...&\eta_{1n}\\
      &\xi_2 &\eta_{23} &...&\eta_{2i} &...&\eta_{2n}\\
      & & \xi_3 &...&\eta_{3i} &... &\eta_{3n}\\
      &&&...&...&...&...\\
       \rowcolor{gray!20}
      &&&\xi{i} &\eta_{i,i+1} &..&\eta_{in}\\
      &&&&&\xi_{n-1} &\eta_{n-1,n}\\
      &&&&&&\xi_{n}
\end{array}\right).
\end{equation}

We design a martingale difference sequence $d_i$,
\begin{align}
\label{eq: mg-diff}
    d_i=E(Z_{n-1}|\Sigma_i)-E(Z_{n-1}|\Sigma_{i-1}),
\end{align}
based on filtration $\Sigma_i$:
\begin{align}
\label{eq:mg-filtration}
\begin{split}
    \Sigma_i &=\sigma\{\xi_1,,\xi_i,\eta_{12},,\eta_{1i},\eta_{23},,\eta_{2i},,,\eta_{i-1,i}\}\\
    &= \sigma \left\{\begin{array}{ccccc}
   \xi_1 & \eta_{12}&... &\eta_{1,i-1} &\eta_{1i}\\
      &\xi_2 &...&\eta_{2,i-1} & \eta_{2i}\\
       & &...&..\\
      & &\xi_{i-2} &\eta_{i-2,i-1},& \eta_{i-2,i}\\
      &&&\xi_{i-1} & \eta_{i-1,i}\\
      &&&&\xi_{i}
\end{array}\right\}.
\end{split}
\end{align}
Then we can show that, 
$$\var \ Z_{n-1} = \sum_{i=1}^{n-1} E d_i^2.$$ 

On the other hand, the expectation of Jackknife estimate is: 
\begin{equation}\label{eq:mg-rewrite}
    E\sum_{i=1}^n(Z_{n,i}-\bar Z_n)^2=(n-1)\frac{E(Z_{n,1}-Z_{n,2})^2}{2}.
\end{equation}

Now, we construct another filtration $\mathcal{A}$ such that $E(Z_{n,1}|\mathcal{A})=E(Z_{n,2}|\mathcal{A})$. In particular, we use: 
\begin{equation}
      \mathcal{A}=\sigma\{\xi_3,\ldots,\xi_{i+1},\eta_{34},\ldots,\eta_{3,i+1},\ldots,\eta_{i,i+1}\}.
\end{equation}
This is essentially $\Sigma_{i+1}$, with the first and second row and columns removed. 
 Define 
 \begin{align*}
     U&=E(Z_{n,1}|\Sigma_{i+1})-E(Z_{n,1}|\mathcal{A})\\
     V&=E(Z_{n,2}|\Sigma_{i+1})-E(Z_{n,2}|\mathcal{A})
 \end{align*}
Using the fact that $E(X^2)=E(E(X^2|\Sigma))\geq E(E[X|\Sigma)^2)$ for some random variable $X$ which is measurable w.r.t to some Sigma field $\Sigma$, we get:
\begin{equation}
\label{eq: mg-result}
E(Z_{n,1}-Z_{n,2})^2 \geq  E(U-V)^2 = 2E(d_i^2)
\end{equation}
The result follows from plugging in Eq~\ref{eq: mg-result} to Eq~\ref{eq:mg-rewrite}. 
\end{proofsketch}
\begin{remark}
Using the aforementioned filtration for graphon models, is also possible to prove another network variant of the Efron-Stein inequality following arguments in \citet{boucheron-lugosi-massart-concentration-chapter}.  This alternative procedure does not require the functional to be invariant to node permutation and allows flexibility with the leave-one-out estimates.  However, the resulting estimate is often not sharp.  See the Supplement Section~\ref{sec:addlth} for more details. 
\end{remark}  


\subsection{Beyond the Efron-Stein inequality}
While the Efron Stein inequality in Theorem~\ref{thm:netefron} is surprising and useful for estimating uncertainty for network statistics, it would be much more satisfying if indeed the jackknife estimate of variance in fact coincided with the true underlying variance, at least asymptotically. We want to draw the attention of the reader to leftmost panel in Figure~\ref{fig:ratio}. The solid black line shows the mean and standard error of the ratio between the jackknife estimate of the variance and the true variance for edge density for a blockmodel and a smooth graphon (details in Section~\ref{sec:exp}), as graph size grows. This figure shows the surprising trend that, in fact, the jackknife estimate is not only an upper bound on the true variance; it is in fact asymptotically unbiased. Our next proposition establishes exactly that.  In what follows, let $Z_n$ denote the edge density (see Section \ref{sec:exp}).

\begin{proposition}\label{prop:degree}
Suppose that $\int_0^1 \int_0^1 w^2(u,v) \ du \ dv < \infty $ and $n\rho_n \rightarrow \infty$. Let $\sigma^2 = \lim_{n \rightarrow \infty} n \cdot \var(Z_{n-1})$.  Then,
\begin{align}\label{eq:prop1}
n \cdot E(\vj \ Z_{n-1}) \rightarrow \sigma^2
\end{align}
\end{proposition}

The proof of the above result involves tedious combinatorial arguments and is deferred to the Supplement Section~\ref{sec:suppEDunbiased}. The above proposition says that, the jackknife estimate of variance of the edge density of a sparse graphon model (see Eq~\ref{eq:sparse-graphon}), in expectation, converges to the true variance. This is a somewhat weak result, since it does not say anything about the jackknife estimate obtained from one network. However, it begs the question, whether a stronger result is true. In fact, in the next section, we prove that for a broad class of count functionals, the jackknife estimate is in fact consistent. This paves the way to the next section, which we start by introducing count functionals.

\subsection{Jackknife Consistency for Count Functionals}
\label{sec:consistency-counts}
In this section, we study the  properties of the jackknife for subgraph counts, which are an important class of functionals in network analysis.  In graph limit theory, convergence of a sequence of graphs can be defined as the convergence of appropriate subgraph frequencies \citep{Lovacz-Graph-Limits}. More practically, subgraph counts have been used to successfully conduct two-sample tests in various settings.  In social networks, for example, the frequency of triangles provides information about the likelihood of mutual friendships/connections and is therefore a useful summary statistic.  

The count functionals that we consider were first studied in \citet{Bickel-Chen-Levina-method-of-moments}.  We will now introduce some notation needed to define these functionals.
Let $G_n$ denote a graph with vertex set $V(G_n)=\{1, 2, \ldots,n\}$ and edge set $E(G_n) \subset V(G_n) \times V(G_n)$.
Let $R\subseteq E(G_n)$ denote the subgraph of interest parameterized by its edge set.  For convenience, we assume $ V(R) = \{ 1,2, \ldots, p \}$.  Furthermore, let $G_n[R]$ denote the subgraph induces by the vertices of $R$.

We consider two different types of count functionals. The first notion that we consider counts exact matches and has the following probability under the sparse graphon model: (Eq~\ref{eq:sparse-graphon}):

\begin{align}
\begin{split}
&P(R) = P(G_n[R] = R)   
\\ &= E\left[\prod_{(i,j) \in R} \rho_n w(\xi_i, \xi_j) \wedge 1  \prod_{(i,j) \in \overline{R}} (1-\rho_n w(\xi_i, \xi_j) \wedge 1)\right]
\end{split}
\end{align}
We also consider the functional $Q(R)$, which provides the probability of an induced subgraph containing the subgraph $R$: 
\begin{align}
Q(R) =  P( R \subseteq G_n(R)) = E\left[\prod_{(i,j) \in R} \rho_n w(\xi_i, \xi_j)\right]
\end{align}

Note that $Q(R)$ is agnostic to the presence of additional edges. When the graph sequence is sparse, $P(R)$ and $Q(R)$ are uninformative, as $P(R),  \ Q(R) \rightarrow 0$.  Let $e = |E(R)|$ and $p = |V(R)|$. Instead,  define the following normalized subgraph frequency:
\begin{align}
\label{eq:normalized-prob}
\tilde{P}(R) = \rho_n^{-e} \ P(R)\qquad \tilde{Q}(R) = \rho_n^{-e} \ Q(R)
\end{align} 
Furthermore, let $Iso(R)$ denote the class of graphs isomorphic to $R$, and $|Iso(R)|$ its cardinality. Our estimator of $\tilde{P}(R)$ is given by:
\begin{align}
\label{eq:normalized-estimator}
\hat{P}(R) = \rho_n^{-e} \frac{1}{{n \choose p} \ |\mathrm{Iso}(R)| } \sum_{S \sim R} \mathbbm{1}(S = G_n[S])
\end{align}

 Similarly, define $\hat{Q}(R)$ as:
\begin{align}
\label{eq:normalized-estimator2}
\hat{Q}(R) = \rho_n^{-e} \frac{1}{{n \choose p} \ |\mathrm{Iso}(R)| } \sum_{S \sim R} \mathbbm{1}(S \subseteq G_n[S])
\end{align}

Due to magnification by $\rho_n^{-e}$, (\ref{eq:normalized-prob}), (\ref{eq:normalized-estimator}), and (\ref{eq:normalized-estimator2}) are not necessarily upper bounded by $1$; nevertheless, they are still meaningful quantities related to subgraph frequencies.


\citet{Bickel-Chen-Levina-method-of-moments} establish a central limit theorem for these functionals under general conditions on the sparsity level and structure of the subgraph. Under analogous conditions, we establish the following consistency result:
\begin{theorem}[Jackknife Consistency for Counts]\label{thm:jkconsistent}
	 Suppose that $R$ is a $p$-cycle or acyclic graph.  Furthermore, suppose that $ \int_0^1 \int_0^1 w^{2e}(u,v) \ du \ dv < \infty$ and $n \rho_n \rightarrow \infty$. Let $\sigma^2 = \lim_{n \rightarrow \infty} n \cdot \mathrm{Var} \ \hat{P}(R)$. Then,
\begin{align}
n \cdot \vj \ \hat{P}(R)  \xrightarrow{P} \sigma^2
\end{align} 
\end{theorem}

Our proof relies on a signal-noise decomposition of the jackknife variance. \citet{Bickel-Chen-Levina-method-of-moments} establish that the variance of a count functional is largely driven by the variance of a U-statistic related to the edge structure of the subgraph.  For this U-statistic component of the decomposition, results for jackknifing U-statistics due to \citet{arvesen1969} may be used to show convergence in probability towards the variance of the corresponding U-statistic.  Since the jackknife is a sum of squares, we are able to decouple the effects of a remainder term and show that it is negligible.  We provide a sketch below, and defer the details to Supplement Section~\ref{sec:suppthm2}.

\begin{proofsketch}
Define density (normalized counts) of R when leaving $i$th node out is $Z_{n,i}={n-1 \choose p}^{-1}\rho_n^{-e}(T-T_i)$, where $T$ is total counts of $R$ in $G_n$, $p$ and $e$ are number of vertices and edges in $R$. Define $Z_n={n \choose p}^{-1}\rho_n^{-e}T$. Then $\vj =\sum_i(Z_{n,i}-\overline{Z_{n}})^2$, $\mathrm{Var} \ \hat{P}(R)= \var \ Z_n$.

Theorem 1 of \citet{Bickel-Chen-Levina-method-of-moments} establishes that $n\var(Z_n)$ converges to a positive constant. Thus we scale $\vj$ by $n$, and decompose  $n\vj$ into
\begin{align}
\label{eq:decompose}
    n\hspace{-1mm }\left[\sum_i(\alpha_i-\aibar)^2\hspace{-1mm }-\hspace{-1mm}2\sum_i(\alpha_i-\aibar)(\beta_i-\bibar) \hspace{-1mm}+\hspace{-1mm}\sum_i(\beta_i-\bibar)^2\hspace{-1mm }\right]\hspace{-1.5mm},
\end{align}
where $\alpha_i= Z_{n,i}-E(Z_{n,i}|\boldsymbol{\xi}_n), \ \beta_i=E(Z_{n,i}|\boldsymbol{\xi}_n)$, $\aibar=\frac{1}{n}\sum_{i=1}^n \alpha_i$, $\bibar=\frac{1}{n}\sum_{i=1}^n \beta_i$ and $\boldsymbol{\xi}_n = (\xi_1, \ldots \xi_n)$.
The term $n\sum_i(\beta_i-\bibar)^2$ corresponds to the signal component discussed before the theorem statement.  

We show in the Supplement that $E\sum_i(\alpha_i-\aibar)^2$ can be further written into $E \left[\sum_{S,T} cov(S,T|\boldsymbol{\xi}_n) \right]$, $\forall$ $S$, $T$ $\sim R$. By \citet{Bickel-Chen-Levina-method-of-moments}, $E \left[\sum_{S,T} cov(S,T|\boldsymbol{\xi}_n)\right] = o(\frac{1}{n})$. $n\sum_i(\alpha_i-\aibar)^2$ is thus negligible by Markov Inequality. The cross term $n\sum_i(\alpha_i-\aibar)(\beta_i-\bibar)$ is also negligible by the Cauchy-Schwartz Inequality. 



\end{proofsketch}

\begin{remark}
 Our theoretical results hold for both notions of subgraph frequencies. However, note that $\tilde{Q}(R)$ is independent of n, but $\tilde{P}(R)$ depends on $n$ and approaches $\tilde{Q}(R)$.  While $\sqrt{n}[\hat{P}(R)- \tilde{P}(R)]$ and $\sqrt{n}[\hat{Q}(R)- \tilde{Q}(R)]$ have the same limiting variance, inference for a fixed target using $\hat{P}(R)$ requires stronger sparsity conditions; namely $\rho_n = \omega(1/\sqrt{n})$. See Section \ref{centering-remark} for a related discussion.
 \end{remark}
 
 \begin{remark}
 Central limit theorems and jackknife consistency can also be shown for more general (cyclic) graphs.  However, in these cases, more stringent sparsity conditions are needed.  
\end{remark}

Now, let $f(G_n)$ denote a function of the vector $(\hat{P}(R_1), \ldots, \hat{P}(R_d))$. Furthermore, let $\nabla f$ denote the gradient of $f$ and $\mu \in \mathbb{R}^d$ the limit of $(\tilde{P}(R_1), \ldots, \tilde{P}(R_d))$ as $n \rightarrow \infty$; it turns out that $\mu$ corresponds to an integral parameter of the graphon related to the edge structure of the subgraph.   We have the following result.
\begin{theorem}[Jackknife Consistency for Smooth Functions of Counts] \label{thm:smooth}
	Suppose that $(R_1, \ldots, R_d)$ are $p$-cycles or acyclic graphs and $n \rho_n \rightarrow \infty$. Let $e^* = \max \{ |E(R_1), \ldots E(R_d) \}$ and suppose that $ \int_0^1 \int_0^1 w^{2e^*}(u,v) \ du \ dv < \infty$.  Furthermore, suppose that $\nabla f$ exists in a neighborhood of $\mu$, $\nabla f(\mu) \neq 0$, and that $\nabla f$ is continuous at $\mu$. Let $\sigma_f^2$ is the asymptotic variance of $\sqrt{n}[f(G_n) - f(E(G_n))]$. Then,
\begin{align*}
n \cdot  \vj \ f(G_n) \xrightarrow{P} \sigma_f^2
\end{align*}
\end{theorem}




\begin{proofsketch}
Let  $Z_{n,i} = (Z_{n,i}(1), \ldots Z_{n,i}(d))$, where $d$ is a constant w.r.t $n$ and each entry corresponds to a count functional with node $i$ removed.  Let $\bar{Z}_{n} = \frac{1}{n}\sum_{i=1}^n Z_{n,i}$.
We use a Taylor expansion of $f(Z_{n,i})$ around $\bar{Z}_{n}$.
\begin{align*}
    f(Z_{n,i})
    &=f\bigl(\bar Z_n) +\nabla f(\mu)^T (Z_{n,i}-\bar{Z}_n)+\underbrace{(\nabla f(\zeta_i)-\nabla f(\mu))^T (Z_{n,i}-\bar{Z}_n)}_{E_i},
\end{align*}
where $\zeta_i = (\zeta_{i1}, \ldots, \zeta_{id})=c_i Z_{n,i}+(1-c_i)\bar Z_n $ for some $c\in [0,1]$. 
Thus, we also have:
\begin{align}\label{eq:smoothdecomp}
\begin{split}
     f(Z_{n,i})-\overline{f(Z_{n,i})}&= \underbrace{\vphantom{E_i-\frac{1}{n}\sum_i E_i} \nabla f(\mu)^T (Z_{n,i}-\bar Z_n)}_{I_i}+\underbrace{ E_i-\frac{1}{n}\sum_i E_i}_{II_i}
\end{split}
\end{align}
We bound $n\sum_i (I_i)^2$ and $n\sum_i (II_i) ^2$ separately. Let $\Sigma$ denotes the covariance matrix of a multivariate U-statistic with kernels $(h_1, \ldots, h_d)$, where each $h_j$ is the kernel corresponding to the count functional in the $j^{th}$ coordinate of the vector $Z_{n}$ (see Eq \ref{eq:edge-structure-kernel} for detailed definition). We can show that 
\begin{align*}
    \left|n \sum_i (I_i)^2 - \df^T \Sigma \df\right|=o_p(1).
\end{align*}
We can also show that $n\sum_i (II_i) ^2$ is also $o_p(1)$.  
Then, let $\mu_n=E[Z_n]$. Note that if one counts subgraphs by an exact match as in~\cite{Bickel-Chen-Levina-method-of-moments} $\mu_n\rightarrow \mu$. If one counts subgraphs via edge matching, $\mu_n=\mu$. Thus, both these types of subgraph densities, which asymptotically have the same limit, can be handled by our theoretical results. 
By Theorem 3.8 in \citet{van2000asymptotic},
\begin{equation*}
    \sqrt{n}(f(Z_n)-f(\mu_n)) \rightsquigarrow N(0,\nabla f(\mu)^T \Sigma \nabla f(\mu) )
\end{equation*}
This shows that the jackknife estimate of variance converges to the asymptotic variance of $f(Z_n)$.
The proof is deferred to the Supplement Section~\ref{sec:suppthm3}.
\end{proofsketch}
\bk

\subsection{A Remark on the use of the Network Jackknife for Two-Sample Testing}
\label{centering-remark}
In principle, the jackknife variance provides a quantification of uncertainty that may be used for many inference tasks.  When the limiting distribution is Normal, one may use a Normal approximation; otherwise, one may use Chebychev's inequality. However, in these cases, the centering is $\theta_{n-1} = E(Z_{n-1})$, which depends on $n$.  Inferences about $\theta_{n-1}$ are often useful for a single graph, but for two-sample testing, issues may arise when comparing networks of different sizes. Probability statements involving a fixed population parameter $\theta$ are needed. To ensure that the jackknife yields valid inferences for an appropriate population parameter, we will need to impose some additional assumptions.  In what follows, let $\{\tau_n\}_{n \in \mathbb{N}}$ denote a sequence of normalizing constants and let $U_{n-1} = \hat{\theta}_{n-1} - \theta$ for some $\theta \in \mathbb{R}$.  We have the following result:  
\bk
\begin{figure*}[h]
\vskip 0.2in
\begin{center}
\vspace{-0.5cm}
\includegraphics[width=  \textwidth,height=0.28\textwidth]{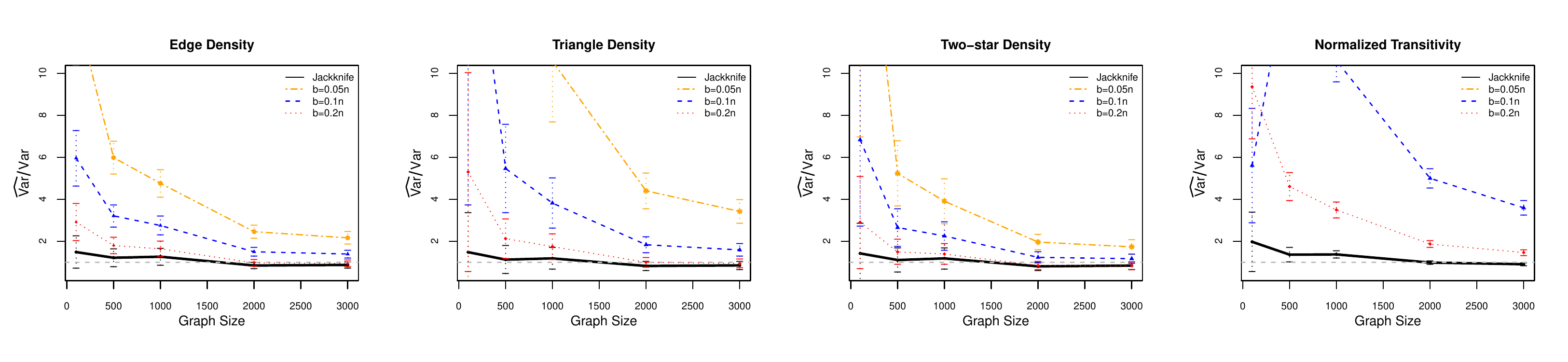}
\begin{footnotesize}
\begin{tabular} {p{0.8cm}p{2.6cm}p{2.6cm}p{2.6cm}p{2.6cm}}
     &(a) &(b) &(c) &(d) \\
\end{tabular}
\end{footnotesize}
\includegraphics[width=  \textwidth,height=0.28\textwidth]{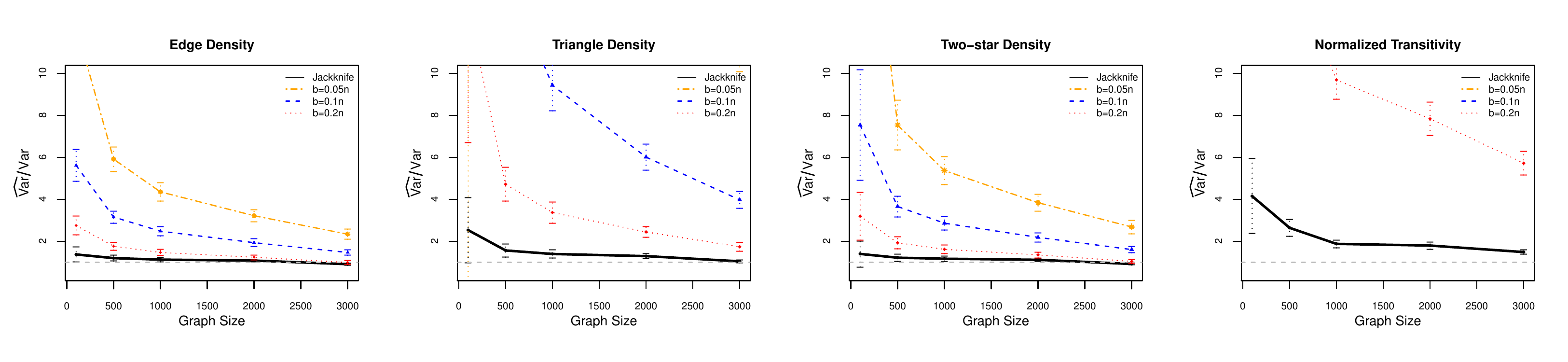}
\begin{footnotesize}
\begin{tabular} {p{0.8cm}p{2.6cm}p{2.6cm}p{2.6cm}p{2.6cm}}
     &(e) &(f) &(g) &(h) \\
\end{tabular}
\end{footnotesize}
\vspace{-0.5cm}
\caption{Ratio of Jackknife estimate $\vj$ to true variance $\var$ for edge density, triangle density, two-star density and transitivity in size $n=100,500,1000,2000,3000$ graphs simulated from the SBM (top) and~\ref{eq:g2} (bottom), compared to subsampling with $b=0.05n$, $b=0.1n$, $b=0.2n$ variance estimation on the same graphs.}
\label{fig:ratio}
\end{center}
\vskip -0.2in
\end{figure*}

\begin{proposition}
\label{prop:centering-result}
Suppose that $\tau_n \rightarrow \infty$ and $\tau_n U_n \rightsquigarrow U$ for some non-degenerate $U$ with mean $0$ and variance $\sigma^2$ and $\{(\tau_n U_n)^2\}_{n \in \mathbb{N}}$ is uniformly integrable. Then,
\begin{align}
\tau_n (\hat{\theta}_n - E(\hat{\theta}_n)) \rightsquigarrow U, \  \frac{\mathrm{Var} \  U_n}{\mathrm{Var} \ \hat{\theta}_n} \rightarrow 1
\end{align}
\end{proposition}

As a consequence of Proposition \ref{prop:centering-result}, if a central limit theorem is known for $\sqrt{n} \ U_n$ and a uniform integrability condition is satisfied, then one may use the jackknife variance in conjunction with a Normal approximation to conduct (possibly conservative) inference for $\theta$.  For count functionals, we have mentioned when this condition holds, so in this case it does not need to be checked. 
\section{Experiments}
\label{sec:exp}
In this section, we present simulation experiments and experiments on real data. For simulations, we compare our variance estimate with that estimated using subsampling. We present our results on two graphons. For real data, we compare networks based on C.I.'s constructed using jackknife estimates of variance of network statistics like edge or triangle density and normalized transitivity.
\paragraph{Count functionals used}
In this section we consider the edge, triangle, two star density, and normalized transitivity, which we explicitly define below.
\begin{align*}
 \text{Edge density} &:=\frac{\sum_{i<j}A_{ij}}{{n \choose 2}\rho_n} 
 \end{align*}
 \begin{align*}
    \text{Triangle density} &:=\frac{\sum_{i<j<k}A_{ij}A_{jk}A_{ki}}{{n \choose 3}\rho_n^{3}} 
     \end{align*}
    \begin{align*}
    \text{Two star density} &:=\frac{\sum_{i,j<k, j,k\neq i}A_{ij}A_{ik}}{{n \choose 3}\rho_n^{2}} 
\end{align*}
As a smooth function of count statistics, we use:
\begin{align*}
    \text{Normalized transitivity}&:=\frac{\text{Triangle density}}{\text{Two star density}}
\end{align*}
\subsection{Simulated Data}
We simulate graphs from two different graphons. The first is a Stochastic Block Model (SBM)~\cite{holland-sbm}, which is a widely used model for networks with communities. A SBM is characterized by a binary cluster membership matrix $Z\in \{0,1\}^{n\times r}$, where $r$ is the number of communities, and a community-community interaction matrix $B$. Conditioned on $Z_{ia}=1$ and $Z_{jb}=1$, nodes $i$ and $j$ form a link with probability $B_{ab}$.
We use  $B=((0.4,0.1,0.1),(0.1,0.5,0.1),(0.1,0.1,0.7))$ and  generate a $Z$ from a $\mbox{Multinomial}(0.3,0.3,0.4)$.  

For the other graphon, we consider the following parameterization:
\begin{align}\label{eq:g2}\tag{GR2}
h_n(u,v) = P(A_{ij} = 1 \ | \ \xi_i = u, \xi_j = v ) = \nu_n |u-v|
\end{align}
where $\nu_n$ is a sparsity parameter. We use $\nu_n=n^{-1/3}$. We denote this graphon by GR2.

From these two graphons, we consider graph size $n$ of  $n=100,500,1000,2000, 3000$.  For each $n$, we simulated 100 graphs to calculate the approximate true variance of edge density, triangle density, two-star density and normalized transitivity among these graphs. 
\paragraph{Computation: }
For each simulated network,  we remove one node at a time, recalculate a statistic  $Z_{n,i}$ on the graph with $(n-1)$ nodes left. Next we compute the jackknife estimate of the variance $\vj:=\sum_i(Z_{n,i}-\bar{Z}_n)^2$, where $\bar{Z}_n$ is the average of the $Z_{n,i}$'s. It should be noted that for some statistics, jackknife, owing to its leave-one-node-out characteristic, can be implemented to reduce computation. For example, in calculating triangles, we calculate the number of triangle on the whole graph once and the number of triangles each node is involved in.  This can be done by keeping track of the number of common neighbors between a node and its neighbors.

For each statistic mentioned, we report the mean of the ratio  $\vj/\var \ Z_n$  among 100 graphs of each $n$ and a 95\% confidence interval from the standard deviation from a normal approximation of these 100 ratios. We also plot a dotted line to denote 1. Closer to this line a resampling procedure is, the better. In Figure~\ref{fig:ratio}, we plot this on the $Y$ axis with $n$ on the $X$ axis. We also plot the same for subsampling with $b=0.05n$, $b=0.1n$, $b=0.2n$ performed on the same graphs.  Figure~\ref{fig:ratio} a,b,c, and d contain results for the SBM, whereas the rest are for the smooth graphon.

We see that for both graphons, $\vj/\var$ converges to $1$ much more quickly in comparison to subsampling and has much smaller variance. These figures also show how susceptible the performance of subsampling is to the choice of $b$. For $b=0.05n$, subsampling overestimates the variance, and exceeds the upper bound on Y axis of some of the figures. In Figure~\ref{fig:ratio} (h) we see that $\vj$ for the normalized transitivity converges slowly for~\ref{eq:g2}, and subsampling with all choices of $b$ are worse as well.


\paragraph{Eigenvalues:}
Here we examine the performance of jackknife on assessing the variance of eigenvalues, to which we have not yet extended our theoretical guarantees.  In Figure~\ref{fig:eig} we show the $\vj/\Var$ for the two principal eigenvalues of the SBM ((a) and (b)) and two graphons described before. Here we only compared with subsampling with $b=0.3n$ and $n=1000,2000,3000$ as subsampling for eigenvalues in sparse graphs only works asymptotically for very large $n$ ~\cite{subsampling-sparse-graphons}. For smaller $n$ and $b$  in our sparsity setting, we saw that subsamples of adjacency matrices often were too sparse leading to incorrect estimate of the variance.
\begin{figure*}[h]
\vskip 0.2in
\begin{center}
\includegraphics[width=  \textwidth,width=  \textwidth,height=0.28\textwidth]{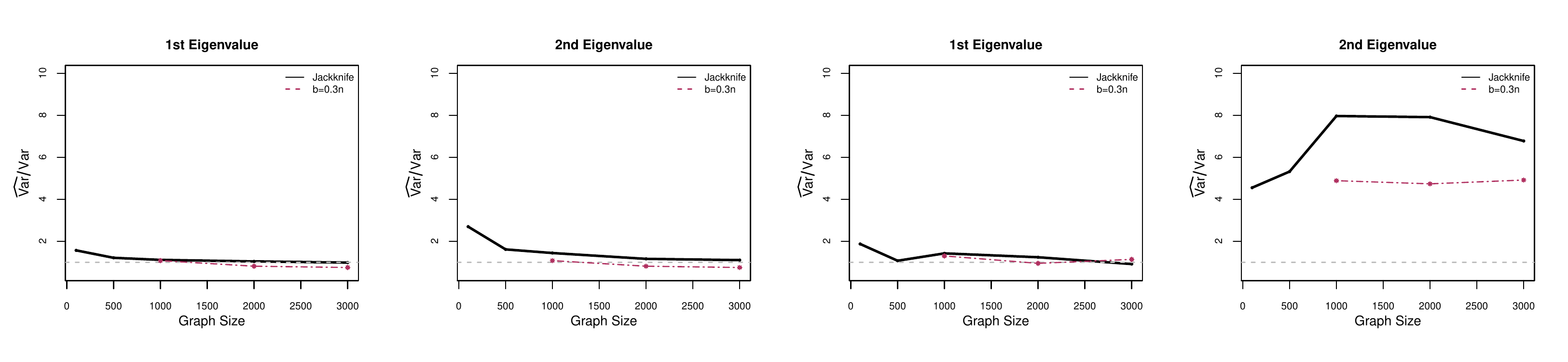}
\begin{footnotesize}
\begin{tabular} {p{0.8cm}p{2.6cm}p{2.6cm}p{2.6cm}p{2.6cm}}
     &(a) &(b) &(c) &(d) \\
\end{tabular}
\end{footnotesize}
\vspace{-0.5cm}
\caption{Ratio of Jackknife estimate $\vj$ to true variance $Var$ for first and second eigenvalues in size $n=100,500,1000,2000,3000$ graphs simulated from stochastic block model in (a) and (b) and the graphon~\ref{eq:g2} in (c) and (d), compared to subsampling with $b=0.3n$ variance estimation on the same graphs.}
\label{fig:eig}
\end{center}
\vskip -0.2in
\end{figure*}
Let us first look at Figure~\ref{fig:eig} (a) and (b) for the SBM setting. For both the eigenvalues in this case, $\vj/\Var$ converges to 1, whereas subsampling consistently underestimates the true variance.  For graphon~\ref{eq:g2}, we see from Figure~\ref{fig:eig} (c) that both jackknife and subsampling estimate the true variance well, whereas for the second eigenvalue (see (d)) they both perform extremely poorly. 
These preliminary results of jackknife estimates show tentative evidence that our theory can be applied to statistics beyond count statistics, like eigenvalues, which we aim to investigate in future work.
\subsection{Real-world Data}
We present two experiments using Facebook network data \cite{fb-data}. In the first experiment, we compared three colleges: Caltech, Williams and Wellesley. While Caltech is known for its strength in natural sciences and engineering, Williams and Wellesley are strong liberal arts colleges. They all have relatively small number of students (800-3000), but with different demographics. For example, Wellesley is a women's liberal arts college, whereas the other have a mixed population. We present the 95\% confidence intervals (CI) obtained using a normal approximation with the estimated variances for two-star and triangle densities for these networks. 

We see that while all three have similar two-star density, Wellesley has significantly higher triangle density. We also see that CI's from jackknife and subsampling with $b=0.1n$ and $b=0.2n$ are comparable. Subsampling with $b=0.05n$ tends have a wider CI, as it overestimates the variance. It is interesting to note that, for triangles, subsampling with $b=0.2n$ took nearly 10 times as much time as jackknife, since we used the leave-one-node-out structure. In comparison, for both methods, two-star counting is overall much faster than counting triangles.



\begin{figure}
\begin{tabular}{cc}
\includegraphics[width=  0.5\textwidth]{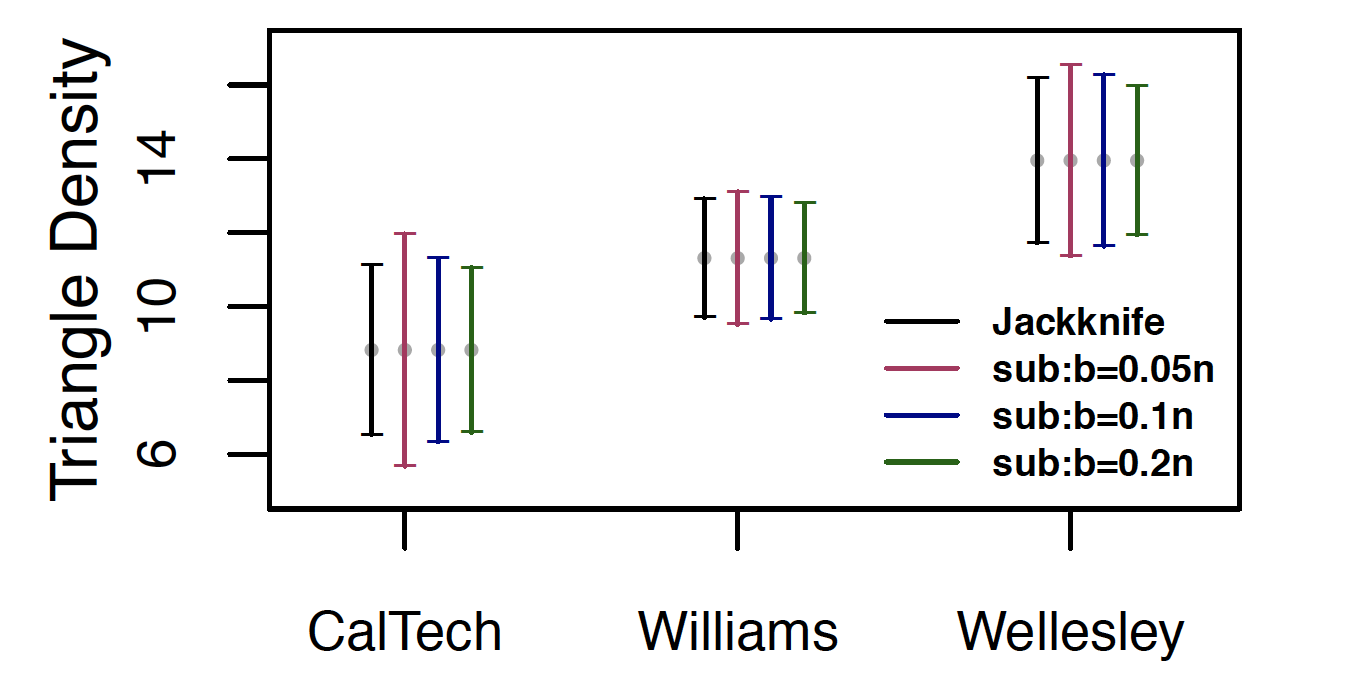}&
\includegraphics[width=  0.5\textwidth]{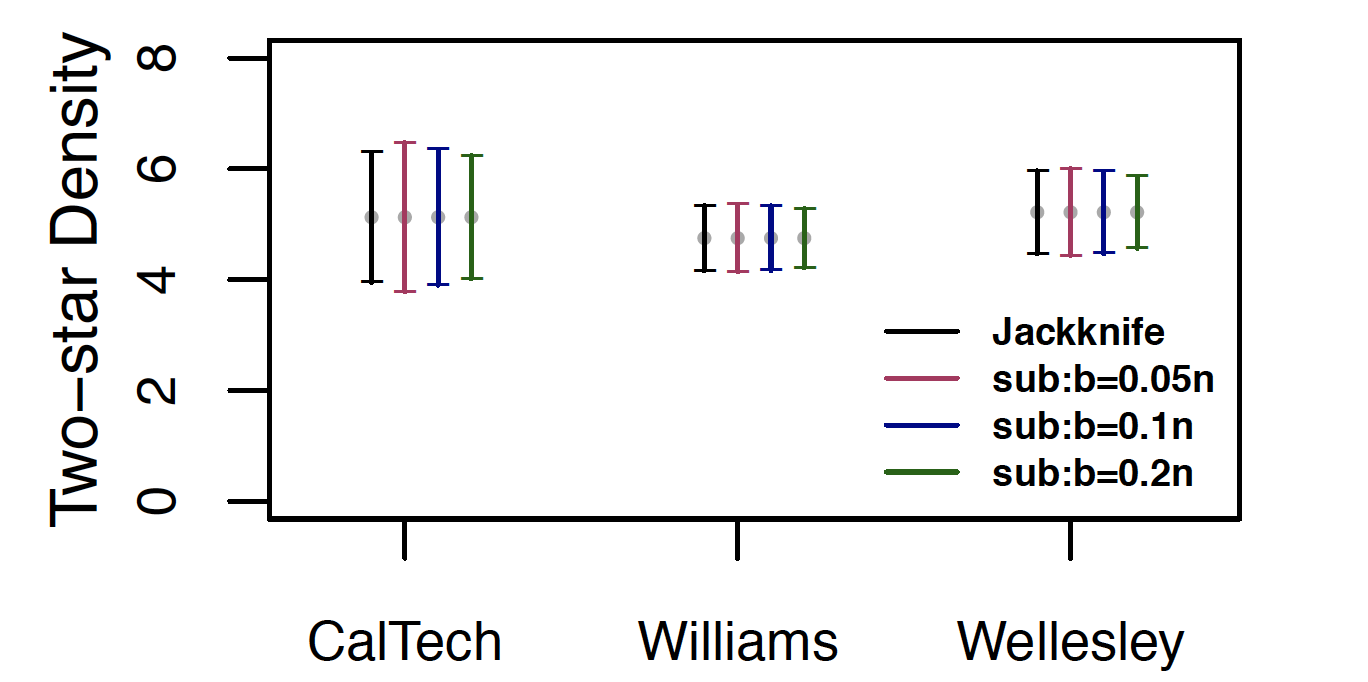}\\
(A)&(B)
\vspace{-2mm}
\end{tabular}
\caption{\label{fig:3colleges} (A) Triangle density , and (B) two-star density (bottom) and their CI's based on jackknife and subsampling variance estimates.}
\vspace{-0.2cm}
\end{figure}

In the second experiment, we look at three college pairs:
Berkeley and Stanford, Yale and Princeton, Harvard and MIT.
First we decide which statistic differentiates between a given pair. For this, we  split each college data set in half, into a training set and test set.  On each of training set, we estimated their triangle density, two-star density, normalized transitivity and their variances estimated by jackknife, demonstrated in Table \ref{table:train}. 
Interestingly, in Table \ref{table:train}, the triangle density is large for all colleges, owing to the sparsity of the networks.  From Table \ref{table:train} we can see normalized transitivity estimates have relatively smaller variance and well separates each of the pairs in training sets. Thus we choose normalized transitivity as the test statistic. We now obtain jackknife estimate of variance of normalized transitivity using the the test sets. 
\begin{table}[t]
	\caption{Triangle, two-star density and normalized transitivity and their variances estimated in college training sets}
	\label{table:train}
	\vskip 0.15in
	\begin{center}
			\begin{tabular}{lccccccr}
				\toprule
				College  &  \multicolumn{2}{c}{Triangle}   &  \multicolumn{2}{c}{Two-star}  & \multicolumn{2}{c}{Norm. Trans.} \\
				& Est & $\widehat{Var}$ & Est &$\widehat{Var}$ & Est & $\widehat{Var}$ \\
				\midrule
				Berkeley    & 77.95 & 18.10 & 6.31 & 0.27 &37.05 & 5.57 \\
				Stanford  & 36.62 & 5.12 & 5.90 &0.11 &18.61 &0.16 \\
				Yale   & 24.20 &2.40 &5.22 &0.09 &13.90 &0.06 \\
				Princeton   & 20.87  &2.34  &5.25 &0.11 &11.91 &0.06 \\
				Harvard    & 38.56 & 5.11 &6.28 &0.10 &18.43 &0.10 \\
				MIT      &  30.20 &7.89 & 6.11 &0.24 &14.82 &0.15 \\
				\bottomrule
			\end{tabular}
\end{center}
\vskip -0.1in
\end{table}

\begin{figure}
	\begin{center}
		\includegraphics[width=  0.5\textwidth]{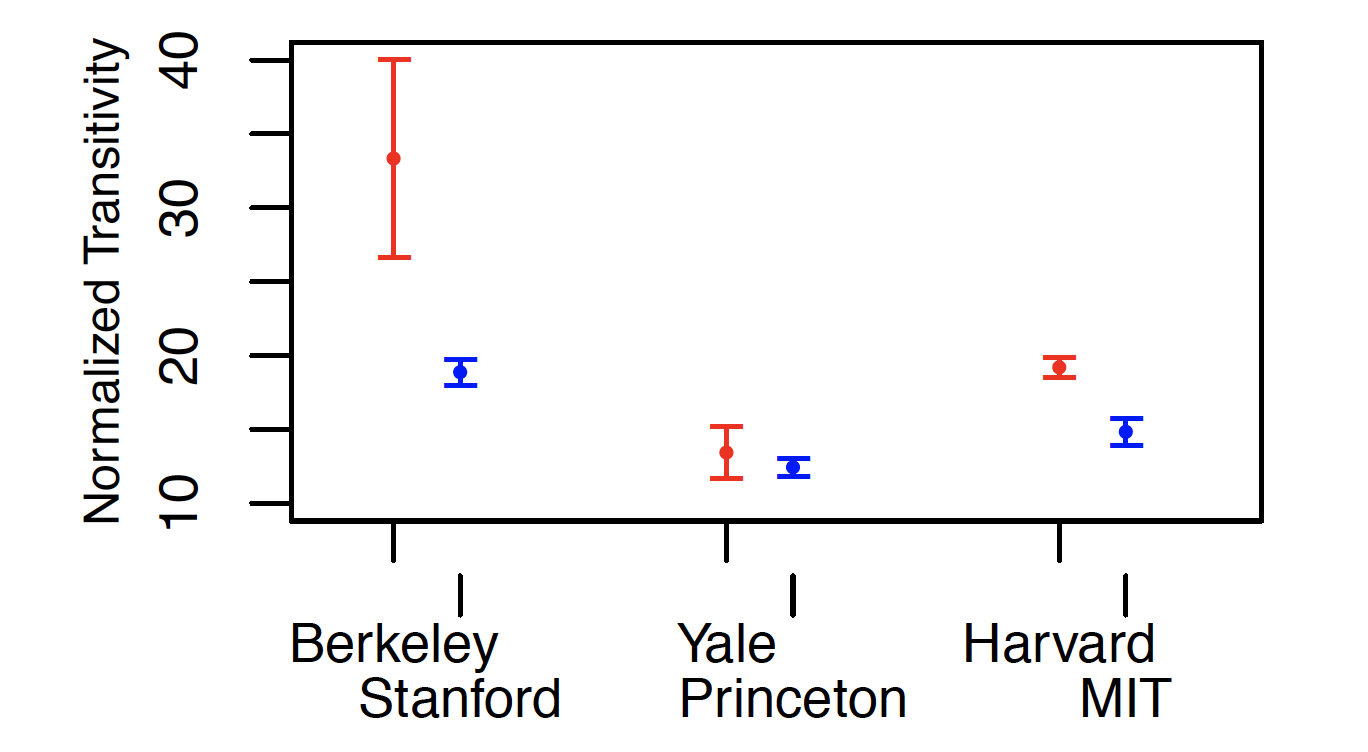}
		\caption{For 3 pairs of colleges, 97.5\% CI constructed using $\vj$ on normalized transitivity }
		\label{fig:6colleges}
	\end{center}
\end{figure}
Figure~\ref{fig:6colleges} presents 97.5\% CI's for normalized transitivity for each college. Thus, two disjoint CI's are equivalent to rejecting a level $0.05$ test. Figure~\ref{fig:6colleges} basically shows that transitivity can in fact separate Berkeley and Stanford Facebook networks, as well as Harvard and MIT Facebook networks, giving us interesting information about the inherent differences between the network structures of these colleges.


\section{Discussion}
In the present work, we have shown that the network jackknife is a versatile tool that may be used in a wide variety of situations.  For poorly understood functionals, the Network Efron-Stein inequality ensures that the jackknife produces conservative estimates of the variance in expectation. For a general class of functionals related to counts, we establish consistency of the jackknife.  Our empirical investigation is encouraging regarding the finite sample properties of the procedure, as the network jackknife outperforms subsampling in many simulation settings.  For certain graphons, our simulations suggest that the jackknife's performance is similar to subsampling,  which will needed be investigated further both theoretically and empirically.     

\begin{appendices}
\renewcommand{\thesection}{A.\arabic{section}}
\renewcommand{\thesubsection}{A.\arabic{subsection}}
\renewcommand{\theequation}{A.\arabic{equation}}
\renewcommand{\thefigure}{A.\arabic{figure}}
\renewcommand{\thetable}{A.\arabic{table}}
\newcommand\Tstrut{\rule{0pt}{2.6ex}}         
\newcommand\Bstrut{\rule[-0.9ex]{0pt}{3ex}}   

\newcommand{\dni}{D^{(n)}_i}
\newcommand{\dnj}{D^{(n)}_j}
\newcommand{\dnip}{D^{(n-1)}_i}
\newcommand{\dnjp}{D^{(n-1)}_j}

\setcounter{equation}{0}
\appendix
\section{Proof of Theorem~\ref{thm:netefron}}
\label{sec:suppES}
To facilitate the proof below, we will explicitly define the data generating mechanism for the Bernoulli trials defined in Eq~\ref{eq:sparse-graphon}.  For $1 \leq i < j \leq n$, define the random variable $\eta_{ij} \sim \mathrm{Unif}[0,1]$ and let $A_{ij} = \mathbbm{1}(\eta_{ij}\leq \rho_n w(\xi_i, \xi_j) \wedge 1)$.  We may view a function $f$ that takes as input a $n-1 \times n-1$ adjacency matrix as a function $g$ of the underlying latent positions.
We require that $g$ is invariant to node-permutation, meaning that $g$ remains unchanged when some (bijective) permutation function $\varphi:\{1,2,\ldots, n-1\} \mapsto \{1,2,\ldots,n-1\}$ is applied to the indices corresponding to $\xi_i$ and both the row and column indices of $\eta_{ij}$ separately. 

In what follows, let $\boldsymbol{\xi}_n = (\xi_i)_{1 \leq i \leq n}$ and $\boldsymbol{\eta}_n = (\eta_{ij})_{1\leq i < j \leq n}$.  Furthermore, we will let $\boldsymbol{\xi}_{n,i}$ denote the vector formed by removing node $i$ and $\boldsymbol{\eta}_{n,i}$ denote the (concatenated) vector formed by removing all elements containing row or column index $i$.     
 
 \begin{proof}
Let $Z_{n,i} = g(\boldsymbol{\xi}_{n,i}, \boldsymbol{\eta}_{n,i})$ denote the functional calculated on an induced subgraph of $n-1$ nodes excluding node $i$.  As before, let $Z_{n-1} = Z_{n,n}$.  
Construct the following martingale difference sequence:
\begin{equation}
    d_i=  E(Z_{n-1}|\Sigma_i)-E(Z_{n-1}|\Sigma_{i-1})
\end{equation}
Here, we consider a filtration introduced by \citet{borgs-convergence-dense-graphons-1}, which was originally used to establish exponential concentration for certain subgraph frequencies in the dense regime.

Let $\Sigma_0 =\{\emptyset, \Omega\}$, $\Sigma_1 = \sigma(\xi_1)$, $\Sigma_2 = \sigma(\xi_1,\xi_2, \eta_{12})$, $\Sigma_3 = \sigma(\xi_1,\xi_2, \xi_3, \eta_{12}, \eta_{13},\eta_{23})$ and so forth up to $n$.  The filtration we consider has the following interpretation: for each time $1 \leq t \leq n$, suppose that we observe a $t \times t$ adjacency matrix induced by the nodes $\{1,2,\ldots,t\}$. Then, $\Sigma_t$ captures all of the randomness in the corresponding induced subgraph.  We may visualize $\Sigma_i$ as a $\sigma$-field generated by a triangular array so that:
\[
\begin{footnotesize}
\Sigma_{i} =\sigma \left\{\begin{array}{ccccc}
  \xi_1 & \eta_{12}&... &\eta_{1,i-1} &\eta_{1i}\\
      &\xi_2 &...&\eta_{2,i-1} & \eta_{2i}\\
      & &...&..\\
      & &\xi_{i-2} &\eta_{i-2,i-1},& \eta_{i-2,i}\\
      &&&\xi_{i-1} & \eta_{i-1,i}\\
      &&&&\xi_{i}
\end{array}\right\}; \ \ 
\Sigma_{i-1} =\sigma \left\{\begin{array}{cccc}
  \xi_1 & \eta_{12}&... &\eta_{1,i-1}\\
      &\xi_2 &...&\eta_{2,i-1}\\
      & &...&..\\
      & &\xi_{i-2} &\eta_{i-2,i-1}\\
      &&&\xi_{i-1} 
\end{array}\right\}
\end{footnotesize}
\]
%

%
\noindent Observe that $Z_{n-1}-E(Z_{n-1})=\sum_{i=1}^{n}d_i$, $d_i$ is $\Sigma_i$ measurable, and  $E(d_i|\Sigma_{i-1})=0$.
 Therefore, the variance of $Z_n$ can be written as:
\begin{align*}
    \var \ Z_{n-1} = E\left(\sum_{i=1}^n d_i\right)^2=\sum_{i=1}^n E(d_i^2) + 2\sum_{i<j} E(d_id_j) 
\end{align*}
Now, for $i \neq j$, observe that:
\begin{align*}
    E(d_id_j)&=E(E(d_id_j|\Sigma_i))=E(d_i)E(d_j|\Sigma_i)\\
    &=E(d_i)(E[E(S_n|\Sigma_j)|\Sigma_i]-E[E(S_n|\Sigma_{j-1})|\Sigma_i])=0
\end{align*}

For the jackknife estimate, we have that: 
\begin{align*}
    E\left(\sum_{i=1}^n(Z_{n,i}-\bar Z_n)^2\right) = \sum_{i<j}\frac{E(Z_{n,i}-Z_{n,j})^2}{n}=\frac{(n-1)\cdot E(Z_{n,1}-Z_{n,2})^2}{2}
\end{align*}

We also denote by $\Sigma_{i:j}$, the sigma field containing all information of random variables $\xi_i,\ldots, \xi_j$, and $\eta_{k\ell}, i\leq k< \ell\leq j$. Now define $\mathcal{A}$ as  $\Sigma_{3:i+1}$.  
Since $Z_{n-1}$ is invariant to node-permutation, $\mathcal{A}$ is independent of $\sigma(\xi_2,\eta_{23}, \ldots,\eta_{2n})$ and $\sigma(\xi_1,\eta_{13},\ldots,\eta_{1n})$,
\begin{align*}
    E(Z_{n,1}|\mathcal{A})=E(Z_{n,2}|\mathcal{A})
\end{align*}
Define:
\begin{align}
    &U=E(Z_{n,1}|\Sigma_{i+1})-E(Z_{n,1}|\mathcal{A}), \quad V=E(Z_{n,2}|\Sigma_{i+1})-E(Z_{n,2}|\mathcal{A})
\end{align}
Then, using the fact that $E[X^2|\Sigma_{i+1}]\geq E[X|\Sigma_{i+1}]^2$ for some $\Sigma_{i+1}$ measurable r.v. $X$, we have:
\begin{equation}
E(Z_{n,1}-Z_{n,2})^2 \geq E[E(Z_{n,1}|\Sigma_{i+1})-E(Z_{n,2}|\Sigma_{i+1})]^2 = E(U-V)^2
\end{equation}

Notice that conditional on $\mathcal{A}$, $U$ is a function of  $\{\xi_2,\eta_{23}, \ldots ,\eta_{2,i+1}\}$, while 
$V$ is a function of $\{\xi_1,\eta_{13},\ldots,\eta_{1,i+1} \}$.
Thus, $U$ and $V$ are conditionally independent. Then, since $\mathcal{A} \subset \Sigma_{i+1}$, by the tower property of conditional expectations, we have that: 
\begin{align*}
E(U-V)^2&=E(U^2)-2E(UV)+E(V^2)=E(U^2)+E(V^2)-2E(E(U|\mathcal{A})E(V|\mathcal{A}))\\
&=E(U^2)+E(V^2),
\end{align*}
Now, we expand $E(U^2)$ as follows: 
\begin{align*}
     E(U^2)&=E((E(Z_{n,1}|\Sigma_{(i+1)})-E(Z_{n,1}|\mathcal{A}))^2]\\
   &\stackrel{(i)}{=}E((E(Z_{n,1}|\Sigma_{2:i+1})-E(Z_{n,1}|\Sigma_{3:i+1}))^2]\\
    &\stackrel{(ii)}{=} E[(E(Z_{n,n}|\Sigma_{1:i})-E(Z_{n,n}|\Sigma_{1:i-1}))^2]
    \\
    &= E[(E(Z_{n-1}|\Sigma_i)-E(Z_{n-1}|\Sigma_{i-1}))^2] \nonumber=E(d_i^2)  \nonumber
\end{align*}
Step $(i)$ holds because the random variables associated with node $1$ are not present in $Z_{n,1}$. Step $(ii)$ holds because $\xi_1, \ldots \xi_n$ and $\eta_{ij}, 1\leq i<j\leq n$ are i.i.d random variables, and $E[Z_{n,1}|\Sigma_{2:i+1}]$ ( $E[Z_{n,1}|\Sigma_{3:i+1}]$ ) and $E[Z_{n,n}|\Sigma_{1:i}]$ ($E[Z_{n,n}|\Sigma_{1:i-1}]$) are equal in distribution.

Similarly, $EV^2=Ed_i^2$, $E(U-V)^2=2Ed_i^2$. Thus,
\begin{equation}
    E(Z_{n,1}-Z_{n,2})^2 \geq E(U-V)^2= 2Ed_i^2
\end{equation}
\begin{equation}
    E\left(\sum_{i=1}^n(Z_{n,i}-\bar{Z}_n)^2\right)=\frac{n-1}{2}E(Z_{n,1}-Z_{n,2})^2\geq(n-1)Ed_i^2=\var \ Z_{n-1}
\end{equation}
 \end{proof}

\section{Proof of Theorem~\ref{thm:jkconsistent}}\label{sec:suppthm2}
 For notational convenience, let $Z_n = \hat{P}(R)$ and let $Z_{n,i}$ denote the subgraph frequency defined in Eq~\ref{eq:normalized-estimator} with node $i$ removed:  
\begin{align}
Z_{n,i} = \rho_n^{-e} \frac{1}{{n-1 \choose p} \ |\mathrm{Iso}(R)| } \sum_{S \sim R, \ i \not \in V(S)} \mathbbm{1}(S = G_n[S])
\end{align}
We first present a lemma that will be used in the proof.  An identity relating the mean of leave-one-out jackknife estimates to a U-statistic plays an important role in the proof of jackknife consistency for U-statistics. Using a novel combinatorial argument, we show that a similar identity holds for normalized subgraph counts: 
\begin{lemma}\label{lem:zn-bar=zn}
Letting $Z_{n,i}$ and $Z_n$ be defined as above, we have that:
\begin{equation*}
    \bar{Z}_n:=\frac{1}{n}\sum_{i=1}^nZ_{n,i}=Z_n
\end{equation*}
\end{lemma}
\begin{proof}
For a subgraph with $p$ nodes and $e$ edges, denote the number of this subgraph in $G_n$ as $Q$. Denote the number of subgraphs node $i$ is involved in as $Q_i$. We now analyze $\sum_{i=1}^n Q_i$. For each vertex set with cardinality $p$, a given subgraph is counted once from each vertex.  Therefore, $\sum_{i=1}^nQ_i = pQ$.  

Observe that $Z_{n,i} + Q_i = Q$ since the set of subgraphs that do not contain node $i$ and the set of subgraphs that contain node $i$ are disjoint and their union gives the set of subgraphs counted in $Q$. It follows that:
\begin{equation*}
    \frac{1}{n}\sum_i Z_{n,i}= \frac{\frac{1}{n}\sum_i(Q-Q_i)}{{n-1 \choose p}\rho_n^e}=\frac{(n-p)Q}{n{n-1 \choose p}\rho_n^e}=\frac{Q}{{n \choose p}\rho_n^e}=Z_n.
\end{equation*}\\
\end{proof}
%
Now, we introduce the limiting value of the scaled variance, which represents the value we are aiming for with the jackknife. \citet{Bickel-Chen-Levina-method-of-moments} show that the  asymptotic behavior of $\hat{P}(R)$ is driven by a U-statistic corresponding to the edge structure of the subgraph.  For a subgraph $R$ with $V(R) = \{1,\ldots, p\}$, define the kernel:
\begin{align}
\label{eq:edge-structure-kernel}
h(x_1, \ldots, x_p) &=  \frac{1}{|Iso(R)|}\sum_{S \sim R, \ V(S) = \{1,\ldots, p\}} \prod_{(i,j)\in E(S)} w(x_i,x_j)    
\end{align}

Theorem 1 of \citet{Bickel-Chen-Levina-method-of-moments} establishes that:
\begin{align*}
n \cdot \mathrm{\var} \ \hat{P}(R) \rightarrow \sigma^2
\end{align*}
where $\sigma^2 = p^2 \zeta$ is the variance of the U-statistic with kernel $h$, with $\zeta = \mathrm{Var}(E(h(\xi_1, \ldots, \xi_p)|\xi_1))$. We will now scale the jackknife variance by $n$ to study its asymptotics. Let:
\begin{align}
\label{eq:ai}
    \alpha_i= Z_{n,i}- E(Z_{n,i}|\boldsymbol{\xi}_n), \quad \beta_i= E(Z_{n,i}|\boldsymbol{\xi}_n)
\end{align}
For simplicity we will use $\aibar$ (or $\bibar$) to denote the average of $\alpha_i$ (or $\beta_i$). 
Now, consider the following signal-noise decomposition:
\begin{align}
  n \cdot \sum_{i=1}^n (Z_{n,i}-\bar{Z}_n)^2  &= 
n \cdot \sum_{i=1}^n(\alpha_i-\aibar+\beta_i-\bibar)^2
    \notag\\ &=  n \cdot \sum_{i=1}^n(\alpha_i-\aibar)^2+2 n \cdot \sum_{i=1}^n(\alpha_i-\aibar)(\beta_i-\bibar)\notag\\
    &+ n \cdot \sum_{i=1}^n(\beta_i-\bibar)^2. \label{eq:signalnoise}
\end{align}

We start by bounding the third sum, which is the signal in our decomposition.  Observe that $\beta_i$ is a U-statistic with the kernel $h$ defined in (\ref{eq:edge-structure-kernel}); therefore, by  Theorem 1 and its following discussions of Chapter 5 in \citet{Lee-Ustats},  we have that:
\begin{align}\label{eq:betaisingle}
n \cdot \sum_{i=1}^n(\beta_i-\bibar)^2 \xrightarrow{P} \sigma^2
\end{align}

The result will follow if we show that the remaining two sums in the decomposition are negligible.  If the first sum is negligible, the Cauchy-Schwarz inequality would imply that:
\begin{align*}
n \cdot \sum_{i=1}^n(\alpha_i-\aibar)(\beta_i-\bibar) & \leq n \cdot \sqrt{\sum_{i=1}^n(\alpha_i-\aibar)^2 \cdot \sum_{i=1}^n(\beta_i-\bar{\beta}_n)^2} \xrightarrow{P} 0 
\end{align*}
It remains to show that: $n \cdot \sum_{i=1}^n(\alpha_i-\aibar)^2 \xrightarrow{P}0$. Now, observe that:
\begin{equation*}
    \sum_{i=1}^n(\alpha_i-\aibar)^2= \sum_{i=1}^n \alpha_i^2 -n \aibar^2
\end{equation*}

Expanding the square for $\sum_{i=1}^n \alpha_i^2$ we have that: 
\begin{align*}
    \sum_{i=1}^n \alpha_i^2 & = \sum_{i=1}^n (Z_{n,i}-E(Z_{n,i}|\boldsymbol{\xi}_n))^2 \\
    & =  \sum_{i=1}^n \begin{pmatrix}n-1 \\ k\end{pmatrix}^{-2} \sum_{S \sim R, \ i \not\in V(S)} (\rho_n^{-e}\psi(S)- W(S))\sum_{T \sim R, \ i \not\in V(T)} (\rho_n^{-e}\psi(T)- W(T))
\end{align*}
where $\psi(S)$ and $W(S)$ are given by:  
\begin{align*}
    \psi(S) &= \frac{1}{|Iso(R)|} \prod_{(i,j) \in E(S), \  S\sim R} A_{ij} \ \times \prod_{(i,j) \in \overline{E(S)}, \  S\sim R} 1-A_{ij},
    \\ W(S)&= \frac{1}{|Iso(R)|} \prod_{(i,j)\in E(S), \ S\sim R} w(\xi_i,\xi_j) \ \times \prod_{(i,j) \in \overline{E(S)}, \  S\sim R} 1-\rho_n w(\xi_i, \xi_j)
\end{align*}
and $\overline{E(S)}$ are $(i,j) \in V(S)\times V(S)$ that are not contained in $E(S)$. 
Now, similar to \citet{Lee-Ustats}, we group elements in the sum based on the number of elements in $V(S)\cap V(T)$. For each $|V(S)\cap V(T)|=c$, there are $n-2p+c$ terms in total.  It follows that:
\begin{align*}
     \sum_{i=1}^n \alpha_i^2 &= \begin{pmatrix}n-1 \\ p \end{pmatrix}^{-2}  \sum_{c=0}^p (n-2p+c) \sum_{|V(S)\cap V(T)|=c}(\rho_n^{-e}\psi(S)- W(S))(\rho_n^{-e}\psi(T)- W(T)) \\
     & = \begin{pmatrix}n-1 \\ p\end{pmatrix}^{-2}  \sum_{c=0}^p (n-2p+c) \sum_{|V(S)\cap V(T)|=c}\gamma(S,T), \ \ \text{say.}
\end{align*}
Now we turn to $n{\aibar}^2$; 
\begin{equation*}
    \aibar= \frac{1}{n} \sum_i Z_{n,i}  -  \frac{1}{n} \sum_i E(Z_{n,i} |\boldsymbol{\xi}_n)\stackrel{(i)}{=} Z_n - E(Z_n|\boldsymbol{\xi}_n) 
\end{equation*}
Equality (i) follows from Lemma~\ref{lem:zn-bar=zn}. 
Now expanding $\aibar^2$ in a similar manner, we have that
\begin{align*}
    \aibar^2=
    \frac{(n-p)^2}{n} \begin{pmatrix}n-1 \\ p\end{pmatrix}^{-2}  \sum_{c=0}^p\sum_{|V(S)\cap V(T)|=c} \gamma(S,T), 
\end{align*}

Then,
\begin{align*}
     n \cdot \sum_{i=1}^n(\alpha_i-\aibar)^2 &= \begin{pmatrix}n-1 \\ p\end{pmatrix}^{-2}  \sum_{c=0}^p \left(n-2p+c-\frac{(n-p)^2}{n}\right) \sum_{|V(S)\cap V(T)|=c}\gamma(S,T) \\
     &=\sum_{c=0}^p \sum_{|V(S)\cap V(T)|=c}\left(c-\frac{p^2}{n}\right)\cdot \begin{pmatrix}n-1 \\ 2\end{pmatrix}^{-2} \gamma(S,T)
\end{align*}
Now, taking expectations, we have that:
\begin{align*}
   &\quad E\left(\begin{pmatrix}n-1 \\ p\end{pmatrix}^{-2}  \sum_{c=0}^p\sum_{|V(S)\cap V(T)|=c} \gamma(S,T)\right)\\
   &=  E\left(\begin{pmatrix}n-1 \\ p\end{pmatrix}^{-2}  \sum_{c=0}^p\sum_{|V(S)\cap V(T)|=c}\biggl(\rho_n^{-e}\psi(S)- W(S)\biggr)\biggl(\rho_n^{-e}\psi(T)- W(T)\biggr) \right)\\ &=E \left[\sum \mathrm{Cov}(S,T|\boldsymbol{\xi}_n)\right] = o\left(\frac{1}{n}\right) 
\end{align*}
where the last line follows from the proof of Theorem 1 of \citet{Bickel-Chen-Levina-method-of-moments}.  

Now, by Markov inequality, we have that 
\begin{align}\label{eq:aiconv}
    n \cdot \sum_{i=1}^n(\alpha_i-\aibar)^2 \xrightarrow{P} 0
\end{align} and the result follows.  


\section{Proof of Theorem~\ref{thm:smooth}} \label{sec:suppthm3}
\begin{proof}
Let  $Z_{n,i} = (Z_{n,i}(1), \ldots Z_{n,i}(d))$, where $d$ is a constant w.r.t $n$ and each entry corresponds to a count functional with node $i$ removed.  Each count functional   may involve subgraphs of different sizes.
We will use a Taylor expansion around $\bar{Z}_{n}$.
\begin{align*}
    f(Z_{n,i})&=f\bigl(\bar Z_n) +\nabla f(\zeta_i)^T (Z_{n,i}-\bar{Z}_n)\\
    &=f\bigl(\bar Z_n) +\nabla f(\mu)^T (Z_{n,i}-\bar{Z}_n)+\underbrace{(\nabla f(\zeta_i)-\nabla f(\mu))^T (Z_{n,i}-\bar{Z}_n)}_{E_i},
\end{align*}
where $\zeta_i = (\zeta_{i1}, \ldots, \zeta_{id})=c_i Z_{n,i}+(1-c_i)\bar Z_n $ for some $c\in [0,1]$. 
Thus, we also have:
\begin{align}\label{eq:smoothdecomp}
\begin{split}
     f(Z_{n,i})-\overline{f(Z_{n,i})}&= \underbrace{\vphantom{E_i-\frac{1}{n}\sum_i E_i} \nabla f(\mu)^T (Z_{n,i}-\bar Z_n)}_{I_i}+\underbrace{ E_i-\frac{1}{n}\sum_i E_i}_{II_i}
\end{split}
\end{align}
For the first part we see that,
\begin{align}\label{eq:I}
    n \sum_i (I_i)^2&=
   n \nabla f(\mu)^T \left(\sum_i(Z_{n,i} - \bar{Z}_n)(Z_{n,i} - \bar Z_n)^T \right)\nabla f(\mu)
\end{align}
We will first show that the inner average of the above expression converges to the covariance matrix of $Z_{n,i}$ (recall that here we are considering a finite dimensional vector). Extending the same argument in Eq~\ref{eq:signalnoise} to finite dimensional $Z_{n,i}$'s (and $\alpha_i$ and $\beta_i$'s defined in Eq~\ref{eq:ai}),
\begin{align*}
&n\sum_i(Z_{n,i} - \bar{Z}_n)(Z_{n,i} - \bar{Z}_n)^T \\
&= n\sum_i \left( (\alpha_i-\aibar)(\alpha_i-\aibar)^T+ (\alpha_i-\aibar)(\beta_i-\bibar)^T+(\beta_i-\bibar)(\alpha_i-\aibar)^T\right.\\
&\left.\qquad \qquad +(\beta_i-\bibar)(\beta_i-\bibar)^T\right)
\end{align*}
By Theorem 9 of \citet{arvesen1969} we have that:
\begin{align}\label{eq:betaimulti}
   n\sum_i (\beta_i-\bibar)(\beta_i-\bibar)^T  \cp \Sigma
\end{align}

Above, $\Sigma$ is the covariance matrix of a multivariate U-statistic with kernels $(h_1, \ldots, h_d)$, where each $h_j$ is the kernel corresponding to the count functional in the $j^{th}$ coordinate of the vector $Z_{n}$ (see Eq \ref{eq:edge-structure-kernel}).
\bk
\bk
Now combining Eq~\ref{eq:betaimulti} with Eq~\ref{eq:I} we see that,
\begin{align}\label{eq:Isq}
    \left|n \sum_i (I_i)^2 - f(\mu)^T \Sigma f(\mu) \right| &\leq \|\df\|^2 n\sum_i \|\alpha_i-\aibar\|^2\notag\\
    &+2n\|\df\|^2\sum_i |(\alpha_i-\aibar)^T (\beta_i-\bibar)|
\end{align}

The first part is $o_p(1)$ by an analogous argument leading to Eq~\ref{eq:aiconv}. For the second part, we see that an application of Cauchy Schwarz inequality gives:
\begin{align*}
    n\sum_i |(\alpha_i-\aibar)^T (\beta_i-\bibar)|\leq \sum_{j=1}^d \sqrt{\left(\sum_i n(\alpha_{i}(j)-\aibar(j))^2\right) \left(n \sum_i (\beta_{i}(j)-\bibar(j))^2\right)}
\end{align*}
The first part inside the square root is $o_p(1)$ due to Eq~\ref{eq:aiconv}, and the second part is $O_p(1)$ by Eq~\ref{eq:betaisingle}. Using this in conjunction with Eq~\ref{eq:Isq} and since $\|\df\|$ is bounded, we see that:
\begin{align*}
    \left|n \sum_i (I_i)^2 - \df^T \Sigma \df\right|=o_p(1)
\end{align*}
All that remains now is to show that part $II_i$ in Eq~\ref{eq:smoothdecomp} is negligible even when summed and multiplied by $n$. First note that $(II_i)^2\leq E_i^2$.
    \begin{align}\label{eq:smoothpart2}
    n\sum_i (II_i)^2 &\leq n\sum_i |(\nabla f(\zeta_i)-\nabla f(\mu))^T (Z_{n,i}-\bar Z_n)|^2\notag\\
    &\leq \max_i\|\nabla f(\zeta_i)-\nabla f(\mu)\|^2\left(n\sum_i (Z_{n,i}-\bar{Z}_n)^T(Z_{n,i}-\bar Z_n)\right)
\end{align}
Theorem~\ref{thm:jkconsistent} shows that the second part in the RHS of Eq~\ref{eq:smoothpart2} is $O_p(1)$. We will now show that the first part is asymptotically negligible.

Observe that:
\begin{align*}
     \max_i \| \zeta_i- \mu \| &\leq \max_i \ c_i\|
      Z_{n,i}- \mu \|+ \max_i \ (1-c_i)\|\bar Z_n-\mu \| 
      \\ & \leq \sqrt{d} \cdot \max_{i,j}|Z_{n,i}(j)- \bar Z_n(j)| + 2\|\bar Z_n-\mu \|  
      \\ & \leq \sqrt{d} \cdot \max_j \sqrt{\sum_{i=1}^n\left(Z_{n,i}(j)- \bar Z_n(j)\right)^2} + 2\| Z_n-\mu \|   
\end{align*}

Above, $\bar{Z}_n=Z_n$ by Lemma~\ref{lem:zn-bar=zn}. The first term on the RHS converges in probability to 0 from our Theorem~\ref{thm:jkconsistent}. By Theorem 1 of~\cite{Bickel-Chen-Levina-method-of-moments}, $\|Z_n - \mu \|$  is also negligible. 
Since $ \max_i \| \zeta_i- \mu \|=o_p(1)$ and $\nabla f$ is continuous at $\mu$, by continuity,  we have that $\max_i\|\nabla f(\zeta_i)-\nabla f(\mu)\|^2= o_p(1)$.  Since the second term on the RHS of Eq~\ref{eq:smoothpart2} is $O_p(1)$ from our previous argument and the first term is $o_p(1)$, it follows that the LHS of Eq~\ref{eq:smoothpart2} is $o_p(1)$.    

Let $\mu_n=E[Z_n]$. Note that if one counts subgraphs by an exact match as in~\cite{Bickel-Chen-Levina-method-of-moments} $\mu_n\rightarrow \mu$. If one counts subgraphs via edge matching, $\mu_n=\mu$. Thus, both these types of subgraph densities, which asymptotically have the same limit, can be handled by our theoretical results. 
By Theorem 3.8 in \citet{van2000asymptotic},
\begin{equation*}
    \sqrt{n}(f(Z_n)-f(\mu_n)) \rightsquigarrow N(0,\nabla f(\mu)^T \Sigma \nabla f(\mu) )
\end{equation*}
This shows that the jackknife estimate of variance converges to the asymptotic variance of $f(Z_n)$.

\end{proof}

\section{Proof of Proposition~\ref{prop:degree}}\label{sec:suppEDunbiased}
Throughout this section, we will use the notation $x_n\asymp y_n$ to denote $x_n=y_n(1+o(1))$.
Before presenting the proof, we present two accompanying lemmas which will be used in the proof of Proposition~\ref{prop:degree}.

\begin{lemma}\label{lem:var}
Denote $\dni$ the degree of node  $i$ in the size $n$ graph. \begin{align*}
\begin{split}
    \sum_{i=1}^{n-1}\var\left(\frac{\dni}{{n-1 \choose 2}\rho_n}\right)&\asymp \frac{4}{n^3}E(\var\sum_{k,k\neq i}w(\xi_i,\xi_k)|\xi_i) \\ &+\frac{4}{n}\var[E (w(\xi_i,\xi_k)|\xi_i)] + O(n^{-2}\rho_n^{-1}).
\end{split}
\end{align*}
\end{lemma}

\begin{lemma}\label{lem:cov}
Denote $\dni$ the degree of node  $i$ in the size $n$ graph. 
\begin{equation*}
    \sum_{i,j,i\neq j}\cov\left(\frac{\dni}{{n-1 \choose 2}\rho_{n}},\frac{\dnj}{{n-1 \choose 2}\rho_{n}}\right) \asymp \frac{4}{n}\times 3 \var(E[w(\xi_i,\xi_j)|\xi_i])  + O(n^{-2}\rho_n^{-1})
\end{equation*}
\end{lemma}

We will use the above to lemmas to prove Proposition~\ref{prop:degree}, which we now present.

\begin{proof}
Denote $D_n$ as the total number of edges in graph $G_n$. By definition,
\begin{equation*}
    Z_n=\frac{D_n}{{n \choose 2}\rho_n}
\end{equation*}
Denote $\dni$ the degree of node  $i$ in the size $n$ graph. 
We have that $E\dni=E\dnj$ for any node pair.  Thus the jackknife estimate of edges for a graph with node $i$ removed  is $D_n$ minus the degree of node $i$. 
Define
\begin{equation}
    \gamma_n={n-1 \choose 2}\rho_n; \quad  \gamma_n'={n-1 \choose 2}\rho_{n-1}
\end{equation}
Then by definition, we have
\begin{equation*}
    Z_{n,i}=\frac{D_n-\dni}{{n-1 \choose 2}\rho_n} = \frac{D_n-\dni}{\gamma_n}
\end{equation*}
Then, the jackknife estimate is
\begin{align}\label{eq:jkLHS}
    E\sum_{i=1}^n(Z_{n,i}-\bar Z_n)^2 &= \frac{1}{2n}\sum_{i\neq j}E(Z_{n,i}-Z_{n,j})^2= \frac{1}{2n}\sum_{i\neq j} E\left(\frac{\dni-\dnj}{\gamma_n}\right)^2 \notag\\ &= \sum_{i=1}^{n-1}\var\left(\frac{\dni}{\gamma_{n}}\right)-\frac{1}{n}\sum_{i\neq j}\cov\left(\frac{\dni}{\gamma_n},\frac{\dnj}{\gamma_n}\right)
\end{align}
\noindent whereas the total number of degrees in a $(n-1)$ graph is $D_{n-1}=\sum_{i=1}^{n-1}\dnip/2$ as each edge is counted 2 times from each node. We first obtain an expression for $\var \ Z_{n-1}$.

\begin{align}\label{eq:jkRHS}
    \var \ Z_{n-1} &=\var\left(\frac{\sum_{i=1}^{n-1}\dnip/2}{{n-1 \choose 2}\rho_{n-1}}\right)=\frac{1}{4}(n-1)  \var\left(\frac{\dnip}{\gamma_{n}'}\right)\\
    &+\frac{1}{4}\sum_{i,j,i\neq j}\cov\left(\frac{\dnip}{\gamma_{n}'},\frac{\dnjp}{
    \gamma_{n}'}\right)
\end{align}


For the second term in the R.H.S of Eq~\ref{eq:jkLHS}, from Lemma ~\ref{lem:cov}, it is easy to check that it is $O(n^{-2})$. Thus scaling  Eq~\ref{eq:jkLHS} by $n-1$ we have,
\begin{align}\label{eq:scaled-jk-LHS}
\begin{split}
    &(n-1)E\sum_{i=1}^n(Z_{n,i}-\bar Z_n)^2 = (n-1)\sum_{i=1}^n\var\left(\frac{\dni}{\gamma_{n}}\right) + O\left(\frac{1}{n}\right)\\
    &=\frac{4}{n^2}E[\var\sum_{k,k\neq i}w(\xi_i,\xi_k)|\xi_i] + 4 \var[E(w(\xi_i,\xi_k)|\xi_i)] +O\left(\frac{1}{n\rho_n}\right)+ O\left(\frac{1}{n}\right)
\end{split}
\end{align}

Plugging in Lemma \ref{lem:cov} into the second term of R.H.S of Eq~\ref{eq:jkRHS} and scaling Eq~\ref{eq:jkRHS} by $n-1$, we have

\begin{align}\label{eq:scaled-jk-RHS}
\begin{split}
&(n-1)\var \ Z_{n-1}\\
&=\frac{1}{n^2}E[\var\sum_{k,k\neq i}w(\xi_i,\xi_k)|\xi_i] + \var[E(w(\xi_i,\xi_k)|\xi_i)] \\ &+ 3\var[E(w(\xi_i,\xi_k)|\xi_i)]+O\left(\frac{1}{n\rho_n}\right)  \\  
&= \frac{1}{n^2}E[\var\sum_{k \neq i}w(\xi_i,\xi_k)|\xi_i] + 4 \var[E(w(\xi_i,\xi_k)|\xi_i)]+O\left(\frac{1}{n\rho_n}\right) 
\end{split}
\end{align}

\noindent The difference between Eqs~\ref{eq:scaled-jk-LHS} and~\ref{eq:scaled-jk-RHS} is:
\begin{equation}
     (n-1)E(Z_{n,i}-\bar Z_n)^2-(n-1)\var \ Z_{n-1}=\frac{3}{n^2}E[\var\sum_{k, k\neq i}w(\xi_i,\xi_k)|\xi_i] + O\left(\frac{1}{n\rho_n}\right).
\end{equation}
Note that, we also have:
\begin{align}
    \frac{1}{n^2}E[\var\sum_{k,k\neq i}w(\xi_i,\xi_k)|\xi_i]=\frac{1}{n}E[\var ( w(\xi_i,\xi_k)|\xi_i)]=O\left(1/n\right)\label{eq:errorterm}
\end{align} 

Eq~\ref{eq:errorterm} establishes Eq~\ref{eq:prop1}. Furthermore, in conjunction with Eqs~\ref{eq:jkRHS} and~\ref{eq:jkLHS}, it also shows that both $(n-1) E\sum_{i=1}^n(Z_{n,i}- \overline{Z_{n}})^2$ and $ (n-1) \var \ Z_{n-1}$ converge to positive constants. This concludes our proof.
\end{proof}


\bk

We now present the proofs of Lemmas~\ref{lem:var} and~\ref{lem:cov}.

\begin{proof}[\textbf{Proof of Lemma~\ref{lem:var}}]

Applying law of total variance,
\begin{equation}\label{eq:vardecomp}
    \sum_{i=1}^{n-1}\var\left(\frac{\dni}{\gamma_{n}}\right)=\sum_{i=1}^{n-1}\var\left[E\left(\frac{\dni}{\gamma_n}\bigg{|}\xi\right)\right] + \sum_{i=1}^{n-1}E\left[ \var\left(\frac{\dni}{\gamma_{n}}\bigg{|}\xi\right)\right].
\end{equation}
We now show that the second term on the RHS of the above equation is small.
\begin{align}
    \sum_{i=1}^{n-1}E\left[ \var\left(\frac{\dni}{\gamma_{n}}\bigg{|}\xi\right)\right] &=\sum_{i=1}^{n-1}E\left[\var\left(\frac{\sum_{j\neq i}A_{ij}}{{n \choose 2}\rho_{n}}\bigg{|}\xi\right)\right]\notag\\
    &=\sum_{i=1}^{n-1}E\left(\frac{\sum_{j\neq i}\rho_n w(\xi_i,\xi_j)(1-\rho_n w(\xi_i,\xi_j))}{{n \choose 2}^2\rho_{n}^2}\right)  \nonumber \\
    &\asymp \sum_{i,j,i\neq j} \frac{\rho_n E[w(\xi_i,\xi_j)]}{n^4\rho_{n}^2}=O(n^{-2}\rho_n^{-1})
\end{align}
For the first term on the RHS of Eq~\ref{eq:vardecomp}, for any fixed $i$, we have:
\begin{align}
\label{eq:rhs2}
\begin{split}
    &\var\left(E\left[\frac{\dni}{\gamma_n}\bigg{|}\xi\right]\right) = \var E\left(\frac{\sum_{k,k\neq i}A_{ik}}{\frac{(n-1)(n-2)}{2}\rho_n}\bigg{|}\xi\right) 
    \asymp \frac{4}{n^4}\var\left(\sum_{k,k\neq i}w(\xi_i,\xi_k)\right)\\
    &\asymp \frac{4}{n^4}E\left(\var\sum_{k,k\neq i}w(\xi_i,\xi_k)|\xi_i\right)+\frac{4}{n^4}\var\left(E\sum_{k, k\neq i}w(\xi_i,\xi_k)|\xi_i\right).
\end{split}
\end{align}
Exchanging the sum and expectation in the second term, we can also write,
\begin{equation}
\label{eq:rhs1}
    \frac{4}{n^4}\var\left(E\sum_{k,k\neq i}w(\xi_i,\xi_k)|\xi_i\right)=\frac{4}{n^2}\var[E (w(\xi_i,\xi_k)|\xi_i)].
\end{equation}
Since Eq~\ref{eq:vardecomp} involves a sum over $n-1$ identical terms, owing to the fact that $\{\xi_i\}$ are i.i.d, we get the result by multiplying Eq~\ref{eq:rhs2} and~\ref{eq:rhs1} by $n-1$.
\end{proof}

\begin{proof}[\textbf{Proof of Lemma~\ref{lem:cov}}]
We decompose the covariance into
\begin{align}\label{eq:covdecomp}
    \sum_{i,j,i\neq j}\cov\left(\frac{\dni}{\gamma_{n}},\frac{\dnj}{\gamma_{n}}\right)&= \sum_{i,j,i\neq j}\cov\left(E\left[\frac{\dni}{\gamma_{n}}\bigg{|}\xi\right],E\left[\frac{\dni}{\gamma_{n}}\bigg{|}\xi\right]\right) \notag\\
    &+ \sum_{i,j,i\neq j}E \left[\cov\left(\frac{\dni}{\gamma_{n}},\frac{\dnj}{\gamma_{n}}\bigg{|}\xi\right)\right] .
\end{align}
The second term on the RHS of the above equation is small as shown before. 
\begin{align*}
\begin{split}
     &\sum_{i,j,i\neq j}E \left[\cov\left(\frac{\dni}{\gamma_{n}},\frac{\dnj}{\gamma_{n}}\bigg{|}\xi\right)\right]  \\&= \sum_{i,j,i\neq j}E\left[\cov\left(\frac{\sum_{k,k\neq i} A_{ik}}{\gamma_n},\frac{\sum_{s,s\neq j} A_{js}}{\gamma_n}\bigg{|}\xi\right)\right] \\
    &\asymp \frac{1}{n^4\rho_n^2}\sum_{i,j}E[\var(A_{ij}|\xi)]\\
    & \asymp \frac{1}{n^2\rho_n^2}\rho_n E[w(\xi_i,\xi_j)]=O(n^{-2}\rho_n^{-1})
\end{split}
\end{align*}
For the first term in Eq~\ref{eq:covdecomp}, for any fixed $i$ and $j$, we have
\begin{align}\label{eq:covdidj}
\begin{split}
     &\cov\left(E\left[\frac{\dni}{\gamma_n}\bigg{|}\xi\right],E\left[\frac{\dnj}{\gamma_n}\bigg{|}\xi\right]\right)\\ &= \cov\left(\frac{\sum_{k}^{ k\neq i}w(\xi_i,\xi_k)\rho_n}{\frac{(n-1)(n-2)}{2}\rho_n},\frac{\sum_{s}^{s\neq j}w(\xi_j,\xi_s)\rho_n}{\frac{(n-1)(n-2)}{2}\rho_n}\right)  \\
    &\asymp \frac{4}{n^4}\cov\left(\sum_{k, k\neq i}w(\xi_i,\xi_k),\sum_{s,s\neq  j}w(\xi_j,\xi_s)\right)  \\
    &= \frac{4}{n^4}\sum_{k, k\neq i}\sum_{s,s\neq j} \cov(w(\xi_i,\xi_k),w(\xi_j,\xi_s)).
\end{split}
\end{align}
Let $S_i=\{i,k\}$, and $S_j=\{j,s\}$ be two pairs containing $i$ and $j$ respectively. Some algebraic manipulation yields,
\begin{equation}
\begin{split}
    \sum_{k, k\neq i}\sum_{s,s\neq j} \cov(w(\xi_i,\xi_k),w(\xi_j,\xi_s))
= \sum_{|S_i\cap S_j|=1}\cov(w(\xi_i,\xi_k),w(\xi_j,\xi_s))\\ +
   \sum_{|S_i\cap S_j|=2}\cov(w(\xi_i,\xi_k),w(\xi_j,\xi_s)).
\end{split}
\end{equation}
In the R.H.S of the above expression, the second summation has $n(n-1)$ terms, whereas the first has $n(n-1)(n-2)$ terms. Furthermore, for $|S_i\cap S_j|=2$, it is easy to see that $\cov(w(\xi_i,\xi_k),w(\xi_j,\xi_s))$ is simply the variance of $\var(w(\xi_i,\xi_k))$ which is positive. For $|S_i \cap S_j|=1$, W.L.O.G. let $S_i=\{i,u\}$ and $S_j=\{j,u\}$. Conditioned on the shared node $\xi_u$, 
\begin{align}\label{eq:cov-var-cond}
\cov(w(\xi_i,\xi_u),w(\xi_j,\xi_u))  &=
\cov[E(w(\xi_i,\xi_u)|\xi_u),E(w(\xi_j,\xi_u)|\xi_u)] \notag\\
&= \var(Ew(\xi_i,\xi_u)|\xi_u) 
\end{align}
which is also positive.
Hence the contribution of the first sum is of a larger order. 

Now we enumerate all the ways in which $S_i$ and $S_j$ can have a node in common, with the constraint of $i\neq j$. For any fixed $i$ and $j$, s.t. $i\neq j$,  $|S_i\cap S_j|=1$ means that there is $1$ common node in  $S_i=\{i,k\}$ and  $S_j=\{j,s\}$.  There are three possible cases, $i=s$, $k=j$, $k=s$. Thus, Eq~\ref{eq:covdidj} can be expanded as (W.L.O.G, suppose $i=s$),
\begin{align}\label{eq:covEij}
     \cov\left(E\left[\frac{\dni}{\gamma_n}\bigg{|}\xi\right],E\left[\frac{\dnj}{\gamma_n}\bigg{|}\xi\right]\right) &\asymp \frac{4}{n^4}[3(n-2)\cov(w(\xi_i,\xi_k),w(\xi_j,\xi_i)) ] \nonumber \\ &=
   \frac{4}{n^3}\times 3 \cov(w(\xi_i,\xi_k),w(\xi_j,\xi_i)) \nonumber \\
   &\stackrel{(i)}{=}\frac{4}{n^3}\times 3 \var(E(w(\xi_i,\xi_k))|\xi_i)  
\end{align}
  
Step $(i)$ uses an analogous argument from Eq~\ref{eq:cov-var-cond}, and conditions on $\xi_i$. 

Eq~\ref{eq:covdecomp} involves a sum over all $(i,j)$ pairs, $i\neq j$, , owing to the fact that $\{\xi_i\}$ are i.i.d, we get the result by multiplying Eq~\ref{eq:covEij} by $n(n-1)$.
\end{proof}

\section{Proof of Proposition~\ref{prop:centering-result}}

Before we state the proof of our result, recall the following well-known relationship between uniform integrability and convergence of moments. See for example, Theorem 25.12 of \citet{billingsley-probability-measure}.
\begin{proposition}
\label{prop:uniform-integrability-moments}
Suppose that $X_n \rightsquigarrow X$ and $\{X_n\}_{n \geq 1}$ is uniformly integrable. Then, $E(X_n) \rightarrow E(X)$.
\end{proposition}

Now we will prove our proposition below:

\begin{proof}
In what follows let $X_n:=\tau_n[\hat{\theta}_n -E(\hat{\theta}_n)]$ and $V_n = \tau_n \cdot U_n$.  Recall that $U_n=\hat{\theta}_n-\theta$.
While our result here is more general, in a jackknife context, $\hat{\theta}_n = Z_n$ following the notation that we use elsewhere. Consider the following decomposition:
\begin{align*}
\tau_n[\hat{\theta}_n - E(\hat{\theta}_n)] = \tau_n[\hat{\theta}_n - \theta] + E(\tau_n[\theta - \hat{\theta}_n]) 
\end{align*}
Since $\{V_n^2\}_{n \geq 1}$ is uniformly integrable, it follows that $\{V_n\}_{n \geq 1}$ is also uniformly integrable.  Therefore, by Proposition \ref{prop:uniform-integrability-moments}, $E(\tau_n[\theta - \hat{\theta}_n]) \rightarrow 0$. By Slutsky's Theorem, it follows that $\tau_n[\hat{\theta}_n - E(\hat{\theta}_n)] \rightsquigarrow U$.  

To show that the variances converge to the same value, observe that $E(X_n^2)$ is given by:
\begin{align*}
E(X_n^2) = E(V_n^2) - (E(V_n))^2
\end{align*}
First, $V_n^2 \rightsquigarrow U^2$ by continuous mapping theorem. Since $\{V_n^2\}_{n \geq 1}$ is uniformly integrable, $E(V_n^2) \rightarrow E(U^2)$ by Proposition \ref{prop:uniform-integrability-moments} again. Finally, $(EV_n)^2 \rightarrow 0$ and the result follows.   
\end{proof}

\section{Additional theory}\label{sec:addlth}

 It should be noted that a similar inequality for a closely related procedure has an even simpler proof.  This alternative procedure does not require the functional to be invariant to node permutation and allows flexibility with the leave-one-out estimates.  However, the resulting estimate is often not sharp. More concretely, let $Z_n$ denote a function of $A^{(n)}$ and let $\widetilde{Z}_{n,i}$ be an arbitrary functional calculated on a graph with node $i$ removed. Consider the following estimator:
 \begin{align}
\vj \ Z_n = \sum_{i=1}^n (Z_n-\widetilde{Z}_{n,i})^2
 \end{align}
 Combining the aforementioned filtration with arguments in \citet{boucheron-lugosi-massart-concentration-chapter} leads to the following inequality:
 \begin{proposition}[Network Efron-Stein, alternative version]
 \begin{align}
\var \ Z_n\leq E( \vj \ Z_n )
 \end{align} 
 \end{proposition}

\section{Additional experiments}
\label{sec:suppexp}

%

We first present Tables~\ref{tab:3college_size} and~\ref{tab:6college_size} with details of the networks we used in our real data experiments in Section~\ref{sec:exp} of the main paper. 
\begin{table}[h]
    \centering
     \caption{\label{tab:3college_size}Details of college networks for first real data experiment (see Figure~\ref{fig:3colleges} of main paper)}
     \vspace{3mm}
    \begin{tabular}{|lccr|}
        \hline 
       &Caltech &Williams &Wellesley  \Tstrut\Bstrut\\
       \hline 
         Nodes &769 &2790 &2970 \Tstrut\Bstrut\\[1ex]
         Edges & 16656& 112986& 94899\\[1ex]
         Average Degree & 43.375 & 63.927 &81.023 \\[1ex]
          \hline
    \end{tabular}
\end{table}
\begin{table}[h]
    \centering
     \caption{\label{tab:6college_size}Details of college networks for second real data experiment (see Figure~\ref{fig:6colleges} of main paper)}
     \vspace{3mm}
    \begin{tabular}{|lcccccr|}
        \hline
      &Berkeley &Stanford &Yale &Princeton &Harvard &MIT \Tstrut \\[1ex]
       \hline
         Nodes &22937  &11621 &8578 &6596 &15126 &6440\Tstrut\\[1ex]
         Edges & 852444&568330&405450&293320&824617&251252\\[1ex]
         Average Degree  &74.332 &97.819&94.544 & 88.952 &109.040& 78.040\\[1ex]
          \hline
    \end{tabular}
    \label{tab:college_size}
\end{table}

For our real data experiments, (Section~\ref{sec:exp} of main paper) we compared subsampling with jackknife on the three colleges (see Figure~\ref{fig:3colleges}). For simplicity, for the second experiment comparing three pairs of college networks (see Figure~\ref{fig:6colleges}),  we only showed the confidence intervals obtained using jackknife. Here, in Figure~\ref{fig:ss-jk-6colleges}, for completeness, we present  confidence intervals for test sets constructed from the six college networks using both jackknife and subsampling with different choices of $b$. This again shows that jackknife CI's mostly are in agreement with those obtained from subsampling. \bk

\begin{figure*}
\begin{center}
\includegraphics[width=\textwidth]{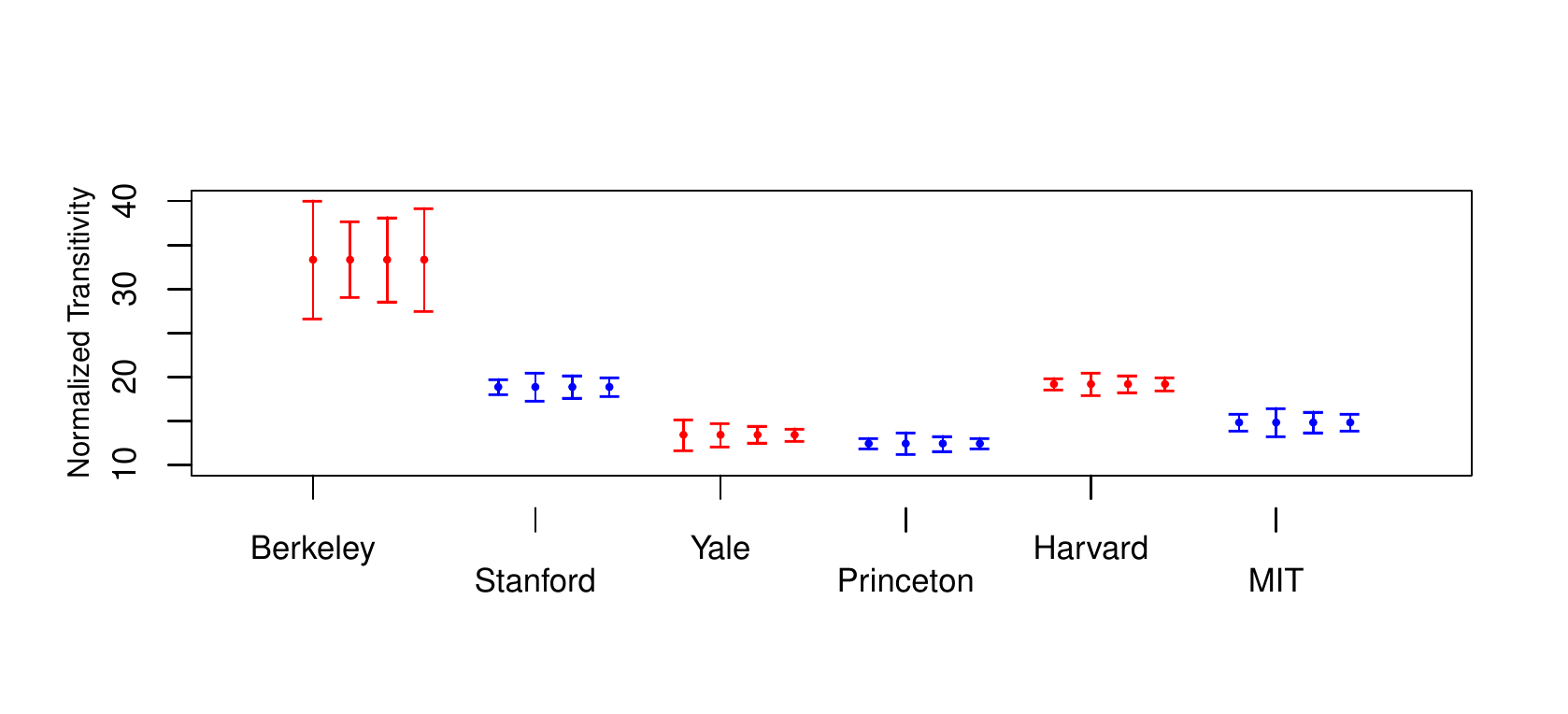}
\vspace{-0.5cm}
\caption{Confidence intervals of subsampling and jackknife in calculating triangle, two-star densities and normalized transitivity in the example of six college Facebook networks test sets. The four CIs for each college are in the order of jackknife, subsampling with b=0.05n, b=0.1n, and b=0.2n respectively.}
\label{fig:ss-jk-6colleges}
\end{center}
\vskip -0.2in
\end{figure*}
\begin{figure*}[h]
\begin{center}
\includegraphics[width=\textwidth]{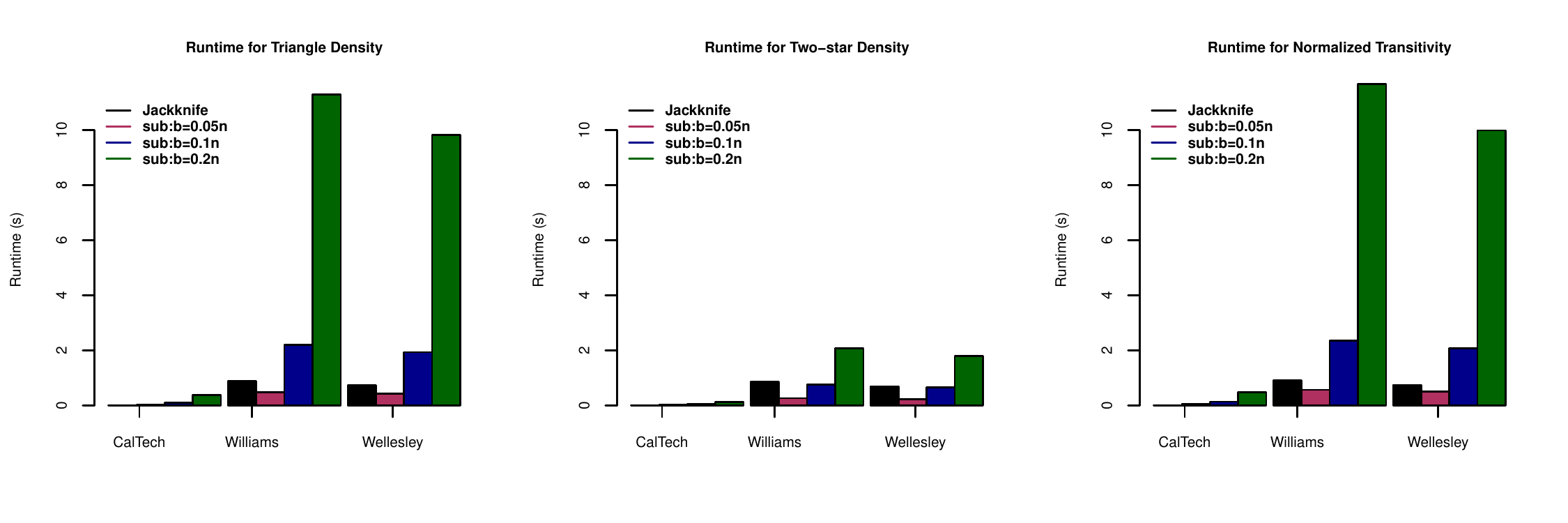}
\vspace{-0.5cm}
\caption{Computation time of jackknife compared to subsampling in calculating triangle, two-star densities and normalized transitivity in the example of three college Facebook networks.}
\label{fig:runtime-3colleges}
\end{center}
\vskip -0.2in
\end{figure*}

In addition, we show the timing results our real data experiments. 
Figure~\ref{fig:runtime-3colleges} shows computation time of the three college example of Facebook network data (see Figure~\ref{fig:3colleges}). We demonstrate the triangle, two-star densities and normalized transitivity variance computation time using jackknife and subsampling with $b=0.05n$, $b=0.1n$ and $b=0.2n$, $B=1000$ in each college network.

In
Figure~\ref{fig:runtime-6colleges}, we show 
the computation time of variance estimation for the same statistics on the test sets for the same set of algorithms. Since we split training and test set in half, the training sets have approximately the same time.

These figures show that, it is possible to implement jackknife in a computationally efficient manner when there is nested structure in the subgraph counts. In all these cases, we see that for the larger networks, subsampling with large $b$ is often considerably slower than jackknife. \bk
\begin{figure*}[t]
\begin{center}
\includegraphics[width=0.85\textwidth]{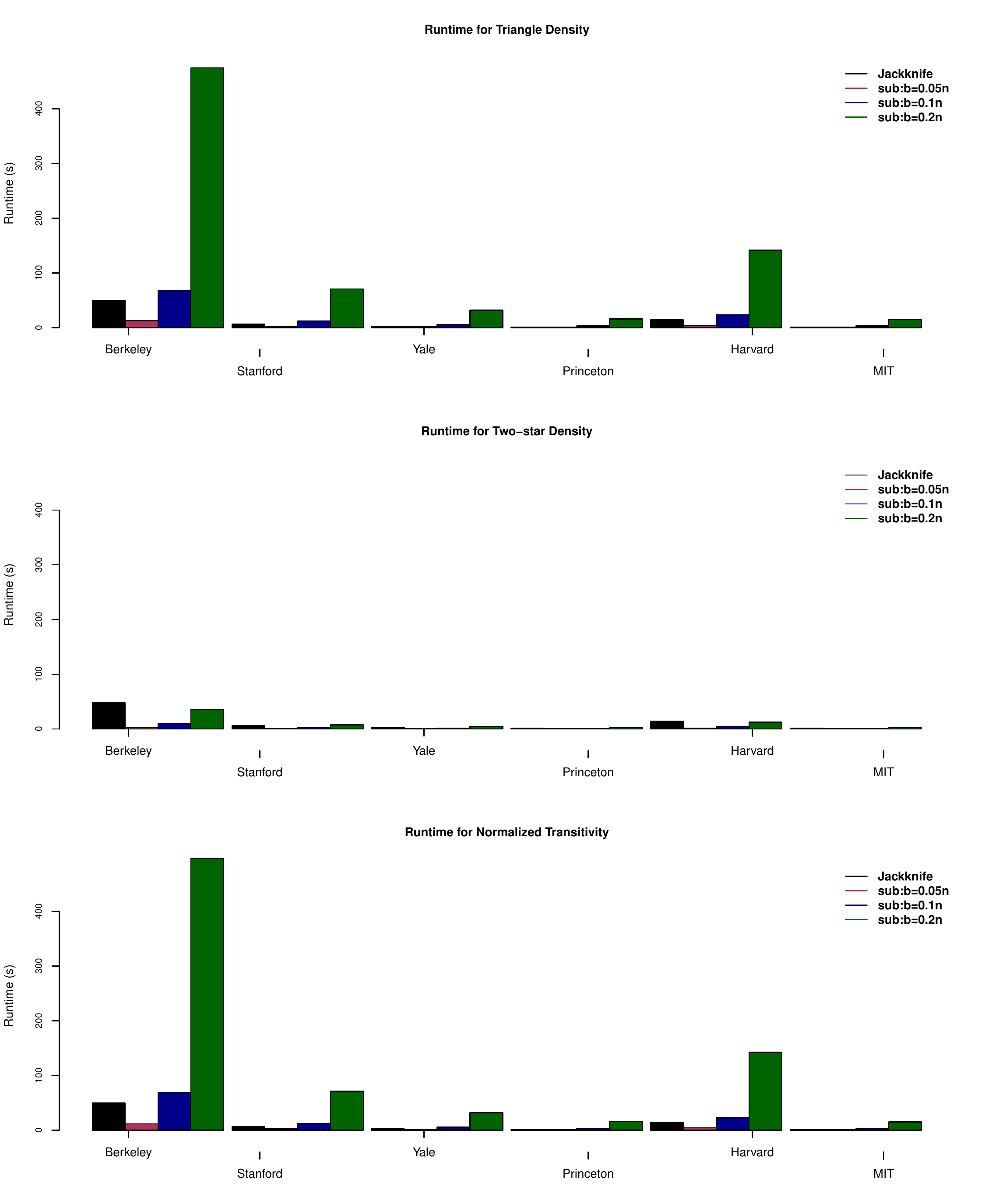}
\vspace{-0.5cm}
\caption{Computation time of jackknife compared to subsampling in calculating triangle, two-star densities and normalized transitivity in the example of six college Facebook networks test sets.}
\label{fig:runtime-6colleges}
\end{center}
\vskip -0.2in
\end{figure*}

\end{appendices}
\clearpage

\bibliographystyle{apalike}
\bibliography{example_paper}

\end{document}